\documentclass[11pt]{article}

\usepackage{amsthm,amsmath,amssymb,bm}
\usepackage{graphicx,color,epsfig}
\usepackage[margin=1in]{geometry}
\usepackage[normalem]{ulem}
\usepackage{authblk}

\newcommand{\N}{\mathbb{N}}
\newcommand{\R}{\mathbb{R}}

\newtheorem{Lem}{Lemma}[section]
\newtheorem{Thm}[Lem]{Theorem}
\newtheorem{Pro}[Lem]{Proposition}
\newtheorem{Claim}[Lem]{Claim}
\newtheorem{Def}[Lem]{Definition}
\newtheorem{Cor}[Lem]{Corollary}

\title{Feedback-delay dependence of the stability of cluster periodic orbits in populations of degrade-and-fire oscillators with common activator}
\author[1]{Bastien Fernandez}
\author[1,2]{Matteo Tanzi}
\affil[1]{Laboratoire de Probabilit\'es, Statistique et Mod\'elisation\\
CNRS - Univ. Paris Cit\'e -  Sorbonne Univ.\\
Paris, France}
\affil[2]{Department of Mathematics, King's College London\\
Strand Building, The Strand\\
London, UK}
\affil[ ]{\textit {fernandez@lpsm.paris, matteo.tanzi@kcl.ac.uk}}
\date{}
\begin{document}
\maketitle

\begin{abstract}
Feedback delay has been identified as a key ingredient in the quorum sensing synchronization of synthetic gene oscillators. While this influence has been evidenced at the theoretical level in a simplified system of degrade-and-fire oscillators coupled via a common activator protein, full mathematical certifications remained to be provided. Here, we prove from a rigorous mathematical viewpoint that, for the very same model, the synchronized degrade-and-fire oscillations are 1/ unstable with respect to out-of-sync perturbations in absence of delay, and 2/ are otherwise asymptotically stable in presence of delay, no matter how small is its amplitude. To that goal, we proceed to an extensive study of the population dynamics in this system, which in particular identifies the mechanisms of, and related criteria for, the delay-dependent stability of periodic orbits with respect to out-of-sync perturbations. As an additional outcome, the analysis also reveals that, depending on the parameters, multiple stable partially synchronized periodic orbits can coexist with the fully synchronized one.
\end{abstract}

\leftline{\small\today.}

\section{Introduction} 
Starting with the toggle-switch and repressilator \cite{GCC00,EL00}, elementary regulatory circuits have been proposed, and implemented, as basic building blocks of gene networks in Synthetic Biology. By understanding (and by controlling) the functioning of simple representative examples, this now popular field of research intends to yield advances in bio-engineering and medical applications,  beyond intrinsic interest to fundamental biology \cite{LMH22}. 

A more recent development of Synthetic Biology aims at investigating the collective dimension of the regulatory dynamics in populations composed by many individuals equipped with simple genetic circuits. In particular, in suitably designed populations of quorum sensing oscillators, stunning evidence of fully synchronized oscillations has been obtained, bringing a standard notion in Physics into the realm of micro-biological colonies \cite{DM-PTH10}.  

This experimental phenomenology has called for theoretical conceptualization based on (simplified) mathematical models for the underlying systems dynamics. In particular, the state of each individual in \cite{DM-PTH10} is represented by the concentration of an auto-repressor protein. Oscillatory behaviour then results from a negative feedback loop for the protein concentration level. The system has been designed so that the production occurs via short but large bursts. Fast production is followed by a slow degradation over longer time intervals. Under appropriate considerations, this {\sl degrade-and-fire} (DF) mechanism can be represented by a differential equation whose vector field is negative constant (corresponding to degradation) and where the concentration is instantaneously reset when it reaches 0 (firing) \cite{MBHT09}. 

In simple models of coupled DF oscillators, rigorous proofs of a sharp transition from a virtually uncoupled regime to massive clustering upon increase of the interaction strength, have been established \cite{FT11,FT14} and confirm appropriate modelling of the phenomenology. 

In a more elaborated model \cite{MHT14}, the quorum sensing mechanism that favours synchrony relies on the presence of a common activator protein that increases the amplitude of the firings. The activator concentration is encoded in an additional variable that is coupled to the mean concentration of the individual repressors. Moreover, as a simplification of the detailed mechanistic model in \cite{DM-PTH10}, a delay has been introduced into the dynamics, which affects the activator concentration involved in the firings. Numerics and theoretical investigations have revealed that this delay plays a crucial role in stabilizing the synchronized oscillations. 

Delayed interactions are known to have a significant impact of the functioning of  biological systems, for instance in the regulation of the synchronization of oscillations in gene networks during development \cite{JST03, YMN20, KIS23}. There exists a large literature on the Lyapunov stability of synchronized oscillations in systems of delayed differential equations \cite{HMN09,JFSJ15,SW89}. The literature also includes instances of systems whose characteristics are close (but distinct) to the degrade-and-fire oscillators and for which stability is shown to sharply depend on the parameter(s) \cite{GJR16,LL22}.  The crucial role played by the delay to favor synchronization has been highlighted in different situations. For example in \cite{J23}, it has been shown how a time-delayed coupling between two oscillators can result into synchronization for arbitrarily small size of the coupling strength (while  perturbation arguments show that synchrony cannot be achieved in absence of delay).

Back to \cite{MHT14}, the authors have carried out a numerical investigation of the typical orbits  arising in their  model under various conditions. In particular, different values of the degradation rate of the concentration of the activator protein, different responses of the reset value on the concentration of the activator, different delays, and different amplitude of the noise (added at each firing event) have been considered. The system exhibits a plethora of behaviours, especially synchronization, stable and metastable clustering, and absence of synchronization. One of the main outcomes of these investigations is that the delay in the coupling has a crucial role in favoring the synchronization of the concentrations. 

The goal of the present paper is to mathematically certify the observations in \cite{MHT14}. More precisely, we prove in particular that for the deterministic dynamics of the model therein, the synchronized oscillations are unstable (with respect to out-of-sync perturbations) in absence of delay, and that they are otherwise asymptotically stable for every positive delay. 

The paper is organized as follows. Firstly, we recall the definition of the model and we provide the basic properties that are useful for the analysis of its dynamics.  Then, we proceed in Section \ref{S-SYNCHRO} to the study of the existence and the delay-dependent stability of the fully synchronized periodic orbit (Proposition \ref{MAINRES}). 
This analysis identifies the main mechanisms involved in that phenomenology. These elements are further developed in Section \ref{S-PARTSYNC} which presents a systematic approach to the stability of arbitrary periodic orbits with partially synchronized repressor concentrations. In particular, a criterion for instability in absence of delay and another criterion for stability in presence of delay are established. The theory culminates with a delay-dependent existence and stability statement (Theorem \ref{STABPARTSYNC}) which is based on the corresponding properties inside the partially synchronized subspace, in absence of delay. Theorem \ref{STABPARTSYNC} is the extension of Proposition \ref{MAINRES} to arbitrary partially synchronized orbits. Finally, an example of application to periodic orbits with several clusters of equi-distributed repressor concentrations is given, which shows in particular that several forms of asymptotically stable (partially) synchronized oscillations can coexist in this system. 

\paragraph{Aknowledgments:} This project has received funding from the European Union’s Horizon 2020 research and innovation
programme under the Marie Sklodowska-Curie grant agreement No 843880.

\section{The dynamical system and its basic properties}
\subsection{Definition of the dynamics}\label{Sec:DefSec}
Following \cite{MHT14}, we consider a population of $N$ DF oscillators ($N\in \N$) represented by the variable $(\mathbf x,A)\in (\R^+)^{N+1}$ where $\mathbf x=(x_1,\cdots ,x_N)\in (\R^+)^N$ collects the concentrations $x_i$ of the repressor proteins and $A$ denotes the activator concentration. 

The dynamics can be depicted as follows. Each repressor concentration decays independently at constant speed -1. When it reaches zero, it is instantaneously reset to a value that depends on the activator concentration. We shall refer to reset events as {\bf firings}. In addition, the repressor proteins contribute to the synthesis of activator, which itself also degrades at constant rate. In formal terms, time variations of the variable $(\mathbf x,A)$ are governed by the following coupled equations
\begin{align}
&\left\{
\begin{array}{lcc}
\dot{x}_i(t)=-1&\text{if}&x_i(t)>0\\
x_i(t^+)=R+\nu A(t-\tau)&\text{if}&x_i(t)=0
\end{array}
\right.\quad\quad\forall i \in\{1,\cdots ,N\}\label{ODE1}\\
&\begin{array}{l}
\dot{A}(t)=m(t)-\beta A(t)\quad\text{if}\quad\ m(t^+)=m(t)\quad\text{where}\quad m(t)=\frac1{N}\sum_{i=1}^Nx_i(t)\label{ODE2}
\end{array}
\end{align}
with initial repressor concentration vector $\mathbf x(0)=\mathbf x^0\in \R^+$ and initial activator concentration profile $A|_{[-\tau,0]}\in (\R^+)^{[-\tau,0]}$. 

The dynamics depends on four parameters,\footnote{The definition in \cite{MHT14} has an additional parameter for the repressors decay rate. However, this rate can be set to 1 by an appropriate rescaling of the variables and other parameters.} namely $R,\beta,\nu\in\R^+_\ast$ and $\tau\in\R^+$ for which, for the sake of the analysis, we impose the conditions\footnote{Notice that the inequalities $\nu<\beta$ and $\tau<R$ hold in all numerical results in \cite{MHT14}.}
\[
\nu<\beta \quad \text{and}\quad \tau<R,
\]
and the inequalities \eqref{CONDORDER} and \eqref{Eq:Cond3} below. Some of the formal statements below explicitly express the dependence on certain parameters, especially $\beta$ and $\tau$. In these cases, all other parameters are implictly assumed to be given beforehand. 

As norms are concerned, both in $\R^N$ and for real functions, we shall use the following notations
\[
\|\mathbf x\|_N=\max_{i\in\{1,\cdots, N\}}|x_i|\quad \text{and}\quad \|A|_I\|_0:=\sup_{t\in I}|A(t)|,
\]
where $N\in\N$ and the interval $I$ are arbitrary.  

\subsection{Basic considerations and elementary properties}
\subsubsection{Existence of global trajectories and conditions for well-posedness in $(\R^+)^{N+1}$}
\paragraph{Existence and uniqueness of global solutions.} Given an arbitrary function $A:[-\tau,+\infty)\to \R^+$, an index $i\in\{1,\cdots ,N\}$ and $x_i^0\in\R^+$, equation \eqref{ODE1} for the single real variable $x_i$, with initial condition $x_i(0)=x_i^0$, trivially admits a unique piecewise linear left continuous solution $x_i:[0,+\infty)\to \R^+$ with constant slope $-1$ and positive jump discontinuities at firings. 

Independently, let $m:[0,+\infty)\to \R^+$ be an arbitrary left continuous piecewise affine function with finitely many discontinuities in every bounded interval and let $A^0\in\R^+$ be arbitrary. Equation \eqref{ODE2} with initial condition $A(0)=A^0$ admits a unique continuous solution $A(t)=\phi^t_m(A^0)$ where $\phi^\cdot_m:[0,+\infty)\to \R^+$ is defined by the following variation of constant formula\footnote{Indeed, equation \eqref{ODE2} determines $\phi_m^t(A^0)$ on a dense subset of $\R^+$, which can then be uniquely extended to the entire $\R^+$ by continuity.} 
\begin{equation}
\phi^t_m(A):=\left(A+\int_0^te^{\beta s}m(s)ds\right)e^{-\beta t}.
\label{Eq:EvforA}
\end{equation}
In addition, this expression implies that $A(t)$ is Lipschitz continuous on every bounded interval. Its Lipschitz constant is controlled by the supremum of $m$ on the interval under consideration together with the value of $A$.

Put together, and since the mean value $\frac1{N}\sum_{i=1}^Nx_i(t)$ associated with the solution must be left continuous and has finitely many jumps on every interval $[0,t]$, the arguments above imply that the coupled equations \eqref{ODE1}-\eqref{ODE2} have, given any initial datum $(\mathbf x^0,A|_{[-\tau,0]})$, a {\bf unique global solution} $(\mathbf x(t),A(t))$ such that $\mathbf x(0)=\mathbf x^0$ and $t\mapsto A(t)$ is continuous on $t\in\R^+$.

From equation \eqref{ODE1} and expression \eqref{Eq:EvforA}, one easily deduces that each oscillator must fire infinitely often in every trajectory. Moreover, the time duration between two consecutive firings of any given oscillator is equal to the reset concentration at the previous firing; in particular this time cannot be smaller than $R$. In addition, the first firing time of oscillator $i$ is $x^0_i$.

\paragraph{Condition for well-posedness in $(\R^+)^{N+1}$.} To provide the profile $A|_{[-\tau,0)}$ in the initial datum only serves to specify the reset concentration(s) at any firing that would occur in the time interval $[0,\tau)$. Actually, since $\tau<R$, only the first firing of some/all oscillators - those firings at the times $x^0_i<\tau$ - can  occur in this interval, because any reset value is at least $R$. Therefore, it suffices to provide ($\mathbf x^0,A^0$ and) the values $A(x^0_i-\tau)$ for $x^0_i<\tau$, in order to define the trajectory, {\sl ie.}\ the dynamics is indeed well-posed in finite dimension.

In particular, for $\tau=0$, the knowledge of $(\mathbf x^0,A^0)$ suffices to define the trajectory; hence the dynamics is well-posed in $(\R^+)^{N+1}$, which is particularly convenient for the stability analysis of periodic orbits. 

When $\tau>0$, if
\[
\text{min}_{\mathbf x^0}:=\min_i x_i^0\geq \tau,
\]
then no firing can occur in the time interval $[0,\tau)$; hence the subsequent trajectory is again well-defined given only $(\mathbf x^0,A^0)$.  
This is also the case when $\text{min}_{\mathbf x^0}<\tau$ if it is impossible that a firing occurs in the past time interval $[\text{min}_{\mathbf x^0}-\tau,0)$. Indeed, one can reverse the time direction in equation \eqref{ODE2} in order to compute the values $A(x_i^0-\tau)$ for $x_i^0<\tau$ using $(\mathbf x^0,A^0)$. More precisely, we have $A(x_i^0-\tau)=\phi_m^{x_i^0-\tau}(A^0)$ where for $t\in\R^+$, $\phi^{-t}_m(A)$ is defined by  the backward time variation of constant formula 
\begin{equation}
\phi^{-t}_m(A):=\left(A-\int_0^te^{-\beta s}m(-s)ds\right)e^{\beta t},
\label{BCKVC}
\end{equation}
with $m(-s)=m^0+s$ for $s\in [0,\tau -\text{min}_{\mathbf x^0}]$ and $m^0:=\frac1{N}\sum_{i=1}^Nx^0_i$. We must also make sure that the values $A(x_i^0-\tau)$ are non-negative, {\sl viz.}\ $A^0\geq A_{\mathbf x^0}(\tau)$, where 
\[
A_{\mathbf x^0}(\tau):=\int_0^\tau e^{-\beta s} (m^0+s)ds,
\]
(which tends to 0 as $\tau\to 0$).

As expression \eqref{Eq:EvforA}, the backward time formula \eqref{BCKVC} is Lipschitz continuous on every bounded interval. Given $\mathbf x, A,\beta$ and $\tau$, let $K:=K_{\mathbf x,A,\beta,\tau}$ be its Lipschitz constant on $[0,\tau]$ when computed with $m(-s)=\frac1{N}\sum_{i=1}^Nx_i+s$. 
\begin{Claim}
Assume that $\text{min}_{\mathbf x^0}\geq \tau$ or $\text{min}_{\mathbf x^0}<\tau$ and
\[
\text{max}_{\mathbf x^0}:=\max_i x^0_i<R+\nu A^0+(\nu K+1)\text{min}_{\mathbf x^0}-(2\nu K+1)\tau\quad\text{and}\quad  A^0\geq A_{\mathbf x}(\tau).
\]
Then the trajectory is well-defined given $(\mathbf x^0,A^0)\in(\R^+)^{N+1}$ and lies in $(\R^+)^{N+1}$. 
\label{CLAIMFINITEDIM}
\end{Claim}
\noindent
{\sl Proof:} The proof is immediate. Together with the Lipschitz continuity mentioned above, the condition in the statement implies 
\[
x_i(t)= \text{max}_{\mathbf x^0}+t<R+\nu A(t-\tau),\ \forall t\in [\text{min}_{\mathbf x^0}-\tau,0],
\]
{\sl ie.}\ no reset value can be attained by the backward flow; hence no firing can occur in the time interval $[\text{min}_{\mathbf x^0}-\tau,0)$. \hfill $\Box$

Notice finally that the last inequality indicates that the following stronger condition
\begin{equation}
\text{max}_{\mathbf x^0}+\text{min}_{\mathbf x^0}<R-\tau\quad\text{and}\quad  A^0\geq A_{\mathbf x}(\tau),
\label{DOMAINFN}
\end{equation}
suffices to obtain the same conclusion. We shall use this sronger condition in the stability analysis of the synchronized periodic orbit in Section \ref{S-PSTABSYNC}. Moreover, notice also that the condition in Claim \ref{CLAIMFINITEDIM} will be necessary for the stability analysis of periodic orbits with equidistributed repressor concentrations (see Section \ref{S-EQUIDIST}).

\subsubsection{Attracting invariant set} 
The dissipative term in equation \eqref{ODE2} suggests that the trajectories should asymptotically approach a bounded forward invariant set.  This property is formally expressed in the next statement. Let  
 \begin{equation}\label{Eq:A0}
 A_\text{max}:= \frac{R}{\beta -\nu}\quad \text{and}\quad Q:=[0,R+\nu  A_\text{max}]^N\times [0, A_\text{max}].
 \end{equation}
 \begin{Lem}
Assume that $\mathbf x^0\in [0,R+\nu  A_\text{max}]^N$ and $A|_{[-\tau,0]}\in [0,A_\text{max}]^{[-\tau,0]}$. Then the solution satisfies 
\[
(\mathbf x(t),A(t))\in Q,\quad \forall t\in\R^+.
\]
In addition, given any initial datum $(\mathbf x^0,A|_{[-\tau,0]})$, the subsequent trajectory satisfies
\[
\limsup_{t\rightarrow +\infty}A(t)\le A_\text{max},\quad\text{and then}\quad\limsup_{t\rightarrow +\infty}x_i(t)\le R+\nu A_\text{max}\quad\forall i\in\{1,\cdots ,N\}.
\] 
\label{Prop:Attractingset}
\end{Lem}
In the rest of the paper, we always assume that $(\mathbf x(t),A(t))\in Q$ for all $t\in\R^+$, even when this is not explicitly stated.
\begin{proof} 
 Assume that $\mathbf x^0\in [0,R+\nu  A_\text{max}]^N$ and $A|_{[-\tau,0]}\in [0,A_\text{max}]^{[-\tau,0]}$. Then, expression \eqref{Eq:EvforA} implies that $A(t)\geq 0$ for all $t\in\R^+$. We prove that $A(t)\leq A_\text{max}$ for all $t\in\R^+$ (and then $\max_i x_i(t)\leq R+\nu A_\text{max}$  for all $t\in\R^+$) by contradiction. Since $t\mapsto \phi_m^t(A(0))$ is continuous, assume otherwise the existence of $\delta,t_\delta\in\R^+_\ast$ such that 
 \[
 A(t_\delta)=A_\text{max}+\delta\quad\text{and}\quad A(t)<A_\text{max}+\delta\ \text{for}\ t\in [-\tau,t_\delta).
 \]
 Then, we certainly have $m(t)<R+\nu (A_\text{max}+\delta)$ for $t\in [0,t_\delta)$. Using expression \eqref{Eq:EvforA}, the definition of $A_\text{max}$ and the condition $\nu<\beta$ successively imply
 \[
 A(t_\delta)\leq \left(A_\text{max}+(R+\nu (A_\text{max}+\delta))\frac{e^{\beta t_\delta}-1}{\beta}\right)e^{-\beta t_\delta}=A_\text{max}+\frac{\nu}{\beta}\delta(1-e^{-\beta t_\delta})<A_\text{max}+\delta
 \]
 which is impossible.
 \medskip
 
 Considering now an arbitrary trajectory, we first prove that for every $t_\ast> \text{max}_{\mathbf x^0}$ such that
 \[
 A(t_\ast)=\max_{t\in [-\tau,t_\ast]}A(t),
 \]
 we must have $A(t_\ast)\leq A_\text{max}$. By contradiction, the fact that all oscillators must have been reset at least once before time $t_\ast$ and the definition of $t_\ast$ imply that we must have $m(t_\ast)\leq R+\nu A(t_\ast)$ and hence
 \[
 \dot A(t_\ast)\leq R+\nu A(t_\ast)-\beta A(t_\ast)<0
 \]
where the second inequality follows from $A(t_\ast)>A_\text{max}$. Moreover, the derivative $\dot A(t)$ is left continuous. Hence $A$ must be decreasing in the left neighbourhood of $t_\ast$, which is impossible from the definition of $t_\ast$. Therefore, we must have $\sup_{t\in\R^+}A(t)<+\infty$ in every trajectory. 

In order to prove that $\limsup_{t\to +\infty}A(t)\le A_\text{max}$, we use a bootstrap argument. Let $\rho\in (0,1)$ be sufficiently large so that $\nu <\rho\beta$. We claim that for every $\delta>0$ such that 
\begin{equation}
A(t)\leq  (1+\delta)A_\text{max},\ \forall t\in [-\tau,+\infty),
\label{BOOTSTRAP}
\end{equation}
there exists $t_\delta\in \R^+$ such that $A(t)\leq (1+\rho\delta)A_\text{max} $ for all $t\in [t_\delta,+\infty)$. Indeed, assume firstly that $A(t)>(1+\rho\delta)A_\text{max}$ for all $t\in\R$. Then we would have 
\[
\dot{A}(t)< -\beta  (1+\rho\delta)A_\text{max} +R+\nu  (1+\delta)A_\text{max}=(\nu -\rho\beta)A_\text{max}\delta<0,\ \forall t\in\R^+,
\]
ie.\ $A(t)$ would have to decrease at least linearly. Given the inequality \eqref{BOOTSTRAP}, it would be impossible that it remains above $(1+\rho\delta)A_\text{max}$ forever. Moreover, a similar reasoning implies that if $A(t_\delta)=(1+\rho\delta)A_\text{max}$, then we must have $A(t)\leq A(t_\delta)$ for all $t>t_\delta$.
\end{proof}

\subsubsection{Preservation of the order in which the oscillators fire} 
Since all degradation rates are equal, in every trajectory, the initial ordering of the repressor concentrations is preserved until the first firing. Under the assumption
\begin{equation}
\nu<\frac{\beta}{1+\beta R},
\label{CONDORDER}
\end{equation}
we are going to show that, when inside $Q$, between any two consecutive firings of a given oscillator, every reset concentration must lie above the current concentration of that oscillator. By induction, this implies that, after the last of the first firing times of each oscillator, the order of the repressor concentrations is cyclically permuted at each firing (and evidently remains constant in time between firings), implying that the order in which the oscillators fire is preserved forever. That property will make the analysis of the dynamics simpler. 

Given $i\in\{1,\cdots ,N\}$ and $k\in\N$, let $t_i^k$ be the instant of the $k^\text{th}$ firing of oscillator $i$. 
\begin{Claim}
Assume that inequality \eqref{CONDORDER} holds and consider a trajectory for which $(\mathbf x(t),A(t))\in Q$ for all $t\in\R^+$. Let $i\in\{1,\cdots ,N\}$ and $k\in\N$ be arbitrary. If an oscillator $j\in\{1,\cdots, N\}$ fires between the $k$th and $(k+1)$th firings of $i$ ({\sl ie.}\ if there exists $\ell\in \N$ such that $t^\ell_j\in (t^k_i,t^{k+1}_i)$), then we have 
\[
x_i\left((t^\ell_j)^+\right)<x_j\left((t^\ell_j)^+\right).
\]
\label{CLAIMORDER}
\end{Claim}
\begin{proof}
 If $(\mathbf x(t),A(t))\in Q$ then we have $|\dot{A}(t)|\leq \beta A_\text{max}$. The assumption \eqref{CONDORDER} then implies $|\dot{A}(t)|<\frac1{\nu}$ which in turn yields
\[
R+\nu A(t_1-\tau)-(t_2-t_1)<R+\nu A(t_2-\tau),\ \forall t_2>t_1,
\]
from where the conclusion is immediate.
\end{proof}

\subsubsection{Return map}\label{S-RETMAP}
In every trajectory, firings must occur infinitely often and their consecutive occurrences are separated by positive time intervals. Moreover, the order preservation obtained in the previous section implies that (for $t>\max_i t_i^1=\text{max}_{\mathbf x^0}$) between any two consecutive firings of oscillator $i$, all other oscillators having repressor concentration distinct from $x_i$ must fire exactly once. Accordingly, in order to analyse the dynamics, it suffices to study the iterations of a {\bf return map} that acts on data immediately before the firing of a given oscillator, say oscillator $N$.

An expression of the return map can be computed as follows. Given $\mathbf x^0$ with $x^0_N=0$, for convenience in the sequel, we denote by $t_\text{R}$, the time $t_N^2=x_N(0^+)=R+\nu A(-\tau)$ of the second firing of $x_N$. We also assume that $\text{max}_{\mathbf x^0}<x_N(0^+)$ so that the order in which the oscillators fire is preserved from $t=0$. Accordingly, for any oscillator for which $x^0_i>0$, the corresponding repressor concentration between $t_i^1=x^0_i$ and $t_\text{R}$ is given by 
\[
x_i(t)=R+\nu A(x^0_i-\tau)-(t-x^0_i),\ \forall t\in (x^0_i,t_\text{R}].
\]
Hence, we have
\[
x_i(t_\text{R})=x^0_i+\nu \left(A(x^0_i-\tau)-A(-\tau)\right),\ \forall i\in\{1,\cdots,N-1\}.
\]
In particular, if $(\mathbf x,A)\in Q$ lies in the Poincar\'e section $x_N=0$ and satisfies the conditions of Claim \ref{CLAIMFINITEDIM}, then the return map $F_N$ writes $(\mathbf x',A')=F_N(\mathbf x,A)$ where
\begin{equation}
\left\{\begin{array}{l}
x'_i=x_{i}+\nu(\phi_m^{x_i-\tau}(A)-\phi_m^{-\tau}(A))\ \text{for}\ i\in \{1,\cdots ,N-1\}\\
A'=\phi^{R+\nu\phi^{-\tau}_{m}(A)}_m(A)
\end{array}\right.
\label{Eq:ReturnMap}
\end{equation}
In the stability analyis of periodic orbits (either fully synchronized or only partially synchronized) in the various sections below, we shall ensure that all iterates $(\mathbf x^k,A^k):=F_N^k(\mathbf x,A)$ of sufficiently small initial perturbations satisfy the conditions of Claim \ref{CLAIMFINITEDIM} (or even condition \eqref{DOMAINFN} in the case of the full synchronized orbit), so that the stability analysis actually reduces to the study of their (linearized) dynamics in $(\R^+)^N$. 

Of note, Appendix \ref{A-LIPCONT} states and proves a certain property of Lipschitz-continuous dependence of the return map on its input datum, not only when the map reduces to one of $(\R^+)^N$ but also in the case of an arbitrary datum $(\mathbf x,A|_{[-\tau,0]})$. This property will be employed in the proof of a stability criterion for partially synchronized periodic orbits in Section \ref{S-PARTSYNC}. 

\section{Existence and stability analysis of the synchronized periodic orbit}\label{S-SYNCHRO}
As a system with mean-field interactions, the equations \eqref{ODE1}-\eqref{ODE2} commute with every permutation of the repressor indices. This suggests to study the {\bf synchronized} dynamics inside the invariant subspace $x_1=x_2=\cdots =x_N$. The synchronized dynamics of the population of $N$ oscillators reduces to that of the $N=1$ system. In this section, we first investigate the corresponding return map $F_1$, a one-dimensional map that acts on the variable $A$ (since we always have $x_1=0$ along every orbit). We show that all orbits asymptotically converge to a unique fixed point $A_\text{FP}$. Then, we study the Lyapunov stability in $(\R^+)^N$ for $N\geq 2$, of the fixed point $(0,\cdots ,0,A_\text{FP})$ of the return map $F_N$. We show that this stability depends on whether the delay-parameter $\tau$ vanishes or it is positive; however and remarkably, it is independent of $N$. The results are collected in the following statement. Letting
\begin{equation}
R_\tau=R+\nu\frac{1+\beta\tau-e^{\beta\tau}}{\beta^2}\quad\text{and}\quad\nu_\tau=\nu e^{\beta\tau},
\label{Eq:DefRtaunutau}
\end{equation}
we will assume the following conditions on the parameters
\begin{equation}
R_\tau>0\quad\text{and}\quad \nu_\tau<\beta. 
\label{Eq:Cond3}
\end{equation}
\begin{Pro}
Assume that \eqref{Eq:Cond3} holds. Then the one-dimensional return map $F_1$ has a unique globally attracting fixed point $A_\text{FP}$.

\noindent
Assume that \eqref{CONDORDER} also holds. 

\noindent
(i) If $\tau=0$, then for every $N\geq 2$, the corresponding fixed point $(0,\cdots,0,A_\text{FP})$ of $F_N$ is unstable in $(\R^+)^N$.

\noindent
(ii) There exists $\tau_0\in \R^+_\ast$ such that for every $\tau\in (0,\tau_0)$ and every $N\geq 2$, the fixed point $(0,\cdots,0,A_\text{FP})$ of $F_N$ is locally asymptotically stable in $(\R^+)^N$.
\label{MAINRES}
\end{Pro}
This statement confirms the numerical observations reported in \cite{MHT14} about the delay-dependence of the stability of the synchronized periodic orbit associated with $A_\text{FP}$.\footnote{In \cite{MHT14}, the repressor resets $R+\nu A(t-\tau)$ are perturbed by a (small) additive random noise. Our result shows that the stability of the synchronized periodic orbit is not affected by these random fluctuations.} An illustration is given in Fig. \ref{TIMESERIES}. 
\begin{figure}[ht]
\begin{center}
\includegraphics*[width=75mm]{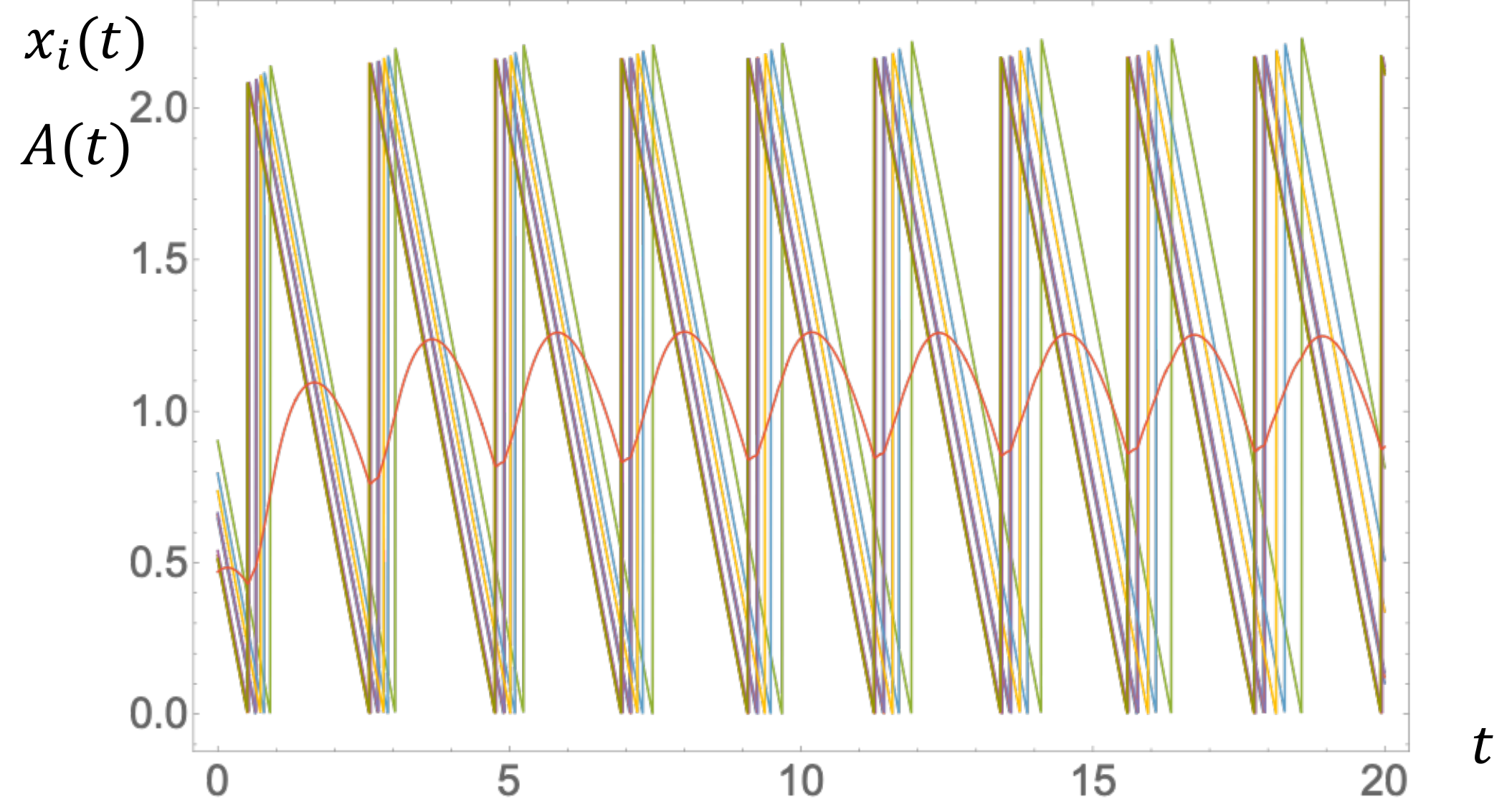}
\hspace{0.9cm}
\includegraphics*[width=75mm]{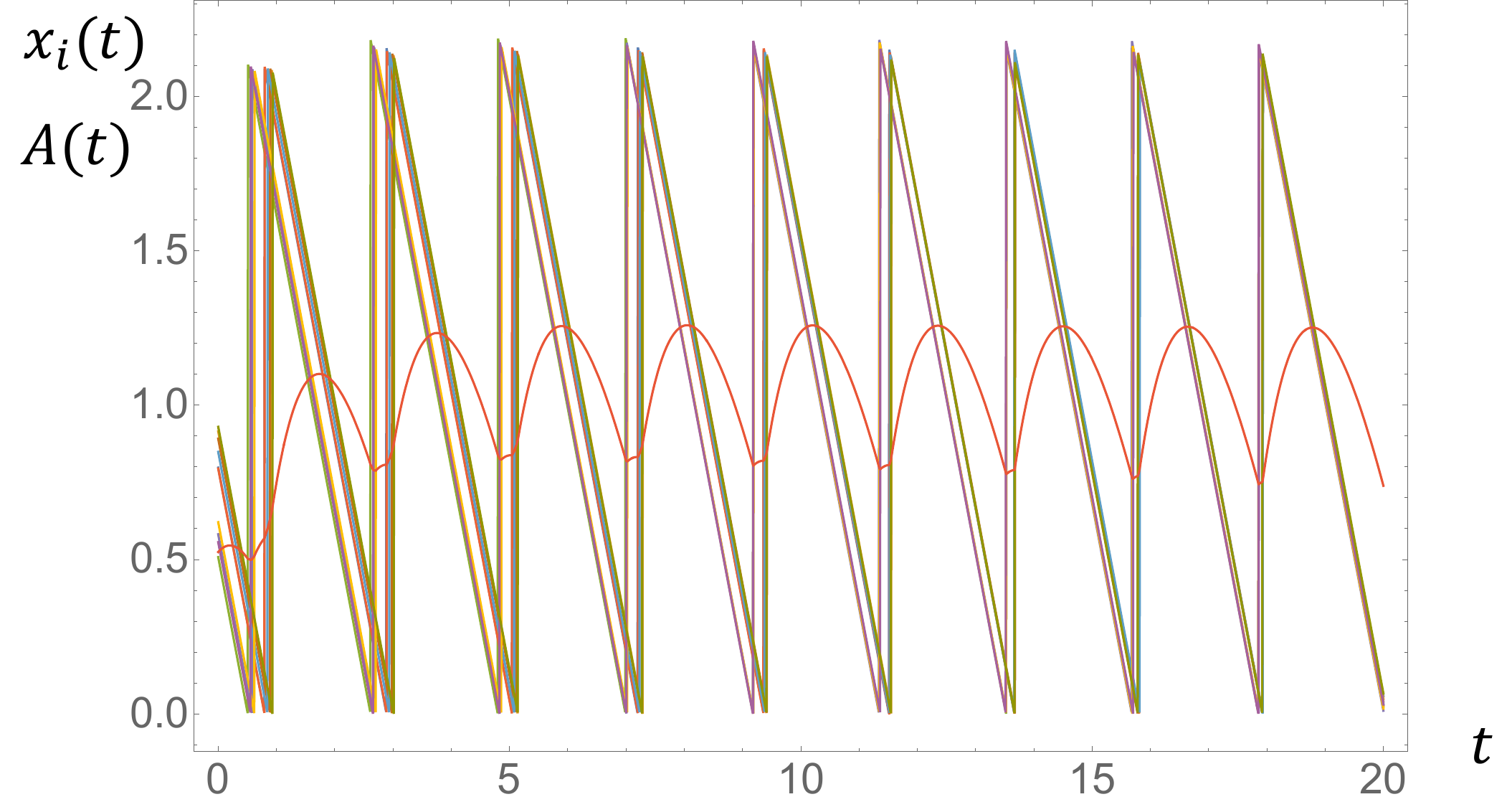}
\end{center}
\caption{Time series of two trajectories for $N=10, t\in [0,20]$ and $\tau=0$ (left)/$\tau=0.2$ (right). The other parameters are $R=2,\beta=1$, and $\nu=0.2$. The sawtooth series correspond to the repressor concentrations $x_i(t)$ ($i\in\{1,\cdots ,10\}$) and the central series in red color corresponds to the activator concentration $A(t)$. On the left picture, the instability of the synchronized periodic orbit is marked although rather weak. On the right picture, the asymptotic stability is more evident.}
\label{TIMESERIES}
\end{figure}

The rest of this section is devoted to the proof of Proposition \ref{MAINRES}.

\subsection{Existence of a globally attracting fixed point of the map $F_1$}
 \begin{figure}
\begin{center}
\includegraphics*[scale=0.35]{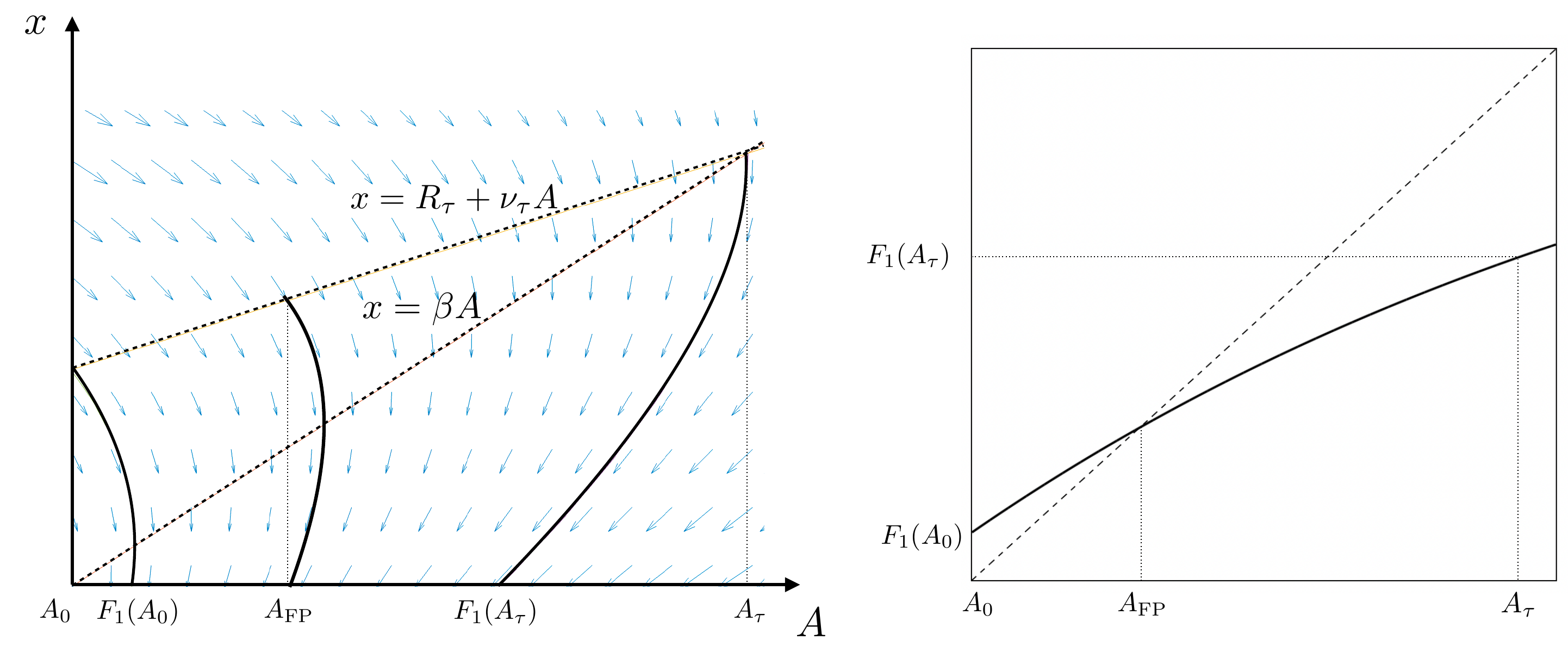}
\end{center}
\caption{\emph{Left.} Illustration of the vector field (light blue arrows) and some segments of trajectories (solid black curves) of the synchronized dynamics/system with one oscillator ($N=1$), under the assumptions \eqref{CONDORDER} and \eqref{Eq:Cond3}. Initial conditions: $(A_0(\tau),R_\tau+\nu_\tau A_0(\tau))$, $(A_\text{FP},R_\tau+\nu_\tau A_\text{FP})$, $(A_\tau,R_\tau+\nu A_\tau)$; where the orbit of $(A_\text{FP},R_\tau+\nu_\tau A_\text{FP})$ is the unique attracting periodic orbit. \emph{Right.} Graph of the corresponding return map $F_{1}$.}
\label{Fig:FlowPlusReturnMap}
\end{figure}

For $N=1$, the condition \eqref{DOMAINFN} reduces to 
\[
A\geq A_0(\tau)=\frac{1-(1+\beta\tau)e^{-\beta\tau}}{\beta^2}.
\]
Moreover, using \eqref{BCKVC} with $m(-s)=s$ for $s\in [0,\tau]$, we obtain that the reset concentration at any firing is given by (NB: see \eqref{Eq:DefRtaunutau} for the definition of $R_\tau$ and $\nu_\tau$ and see Fig.\ \ref{Fig:FlowPlusReturnMap} for an illustration)
\begin{equation*}
R+\nu\phi^{-\tau}_m(A)=R_\tau+\nu_\tau A,
\end{equation*}
Therefore, using \eqref{Eq:EvforA} with $m(s)=R_\tau+\nu_\tau A-s$ for all $s\in (0,R_\tau+\nu_\tau A]$ yields the following explicit expression for the return map $F_1(A)=\phi^{R+\nu\phi^{-\tau}_m(A)}_m(A)$
\[
F_1(A)=\frac1{\beta^2}+\left(A(1-\frac{\nu_\tau}{\beta})-\frac{R_\tau}{\beta}-\frac1{\beta^2}\right)e^{-\beta(R_\tau+\nu_\tau A)}.
\]
Letting  $A_\tau:=\frac{R_\tau}{\beta-\nu_\tau}$, the assumption \eqref{Eq:Cond3} implies $A_0(\tau)<A_\tau$. Clearly, this assumption also implies that the asymptotic dynamics of $F_1$ must lie in the interval $[A_0(\tau),A_\tau]$, again see Fig.\ \ref{Fig:FlowPlusReturnMap} for illustration. The following features then yield that $F_1$ must have a unique globally attracting fixed point inside this interval, proving the preliminary claim in Proposition \ref{MAINRES}:
\begin{itemize}
\item $F_1(A_0(\tau))=\frac1{\beta^2}-\left(\frac{(1+\beta\tau )}{\beta^2}e^{-\beta \tau}+\frac{R}{\beta}\right)e^{-\beta R}>A_0(\tau)$ for every $\tau\in [0,R)$.
\item $F_1(A_\tau)=\frac1{\beta^2}\left(1-e^{-\beta^2A_\tau}\right)<A_\tau$, so that altogether $F_{1}([A_0(\tau),A_\tau])\subsetneq [A_0(\tau),A_\tau]$.
\item $F_1'(A_\tau)>0$ and $F_1''|_{[0,A_\tau]}<0$, viz.\ $F_1|_{[0,A_\tau]}$ is increasing with decreasing derivative.
\end{itemize}

\subsection{Proof of instability for $\tau=0$}
In this section, we prove item {\em (i)} in Proposition \ref{MAINRES}. Consider the restriction of $F_N$ to the subspace of 2-cluster states for which $N-1$ repressor concentrations are equal, more precisely, the restriction to those $\mathbf x$ for which $x_1>0$ and $x_2=\cdots =x_N=0$. We are going to show that $(0,\cdots ,0,A_\text{FP})$ is linearly repelling along the direction of $x_1$. 

Letting $\epsilon>0$ sufficiently small, assume that 
\[
x_1\in (0,\epsilon)\ \text{and}\ |A-A_\text{FP}|<\epsilon.
\]
In particular, $\epsilon$ must be small enough so that such $(\mathbf x,A)$ satisfies the condition \eqref{DOMAINFN}.

Given that  $x_2=\cdots =x_N=0$, we have for the continuous time trajectory
\[
m(t)=\frac{(N-1)(R+\nu A)+x_1}{N}-t,\ \forall t\in (0,x_1],
\]
which yields using \eqref{Eq:EvforA} and after simple algebra
\[
\phi^{x_1}_{m}(A)=A+x_1\left(\frac{N-1}{N}(R+\nu A)-\beta A\right)+ {\cal O}(x_1^2).
\]
From \eqref{Eq:ReturnMap}, we obtain the following expansion for the coordinate $x'_1$ of the first iterate $(\mathbf x',A')=F_N(\mathbf x,A)$
\[
x'_1=x_1+\nu(\phi^{x_1}_{m}(A)-A)=x_1\left(1+\nu\left(\frac{N-1}{N}(R+\nu A)-\beta A\right)\right)+ {\cal O}(x_1^2).
\] 
Therefore, in order to prove the desired instability, all we have to show is $\frac{N-1}{N}(R+\nu A)-\beta A>0$ for all $|A-A_\text{FP}|<\epsilon$ with $\epsilon$ small. By continuity, it suffices to show that $\frac{N-1}{N}(R+\nu A_\text{FP})-\beta A_\text{FP}>0$.
Since $\frac{N-1}{N}\geq \frac12$ for all $N\geq 2$, it suffices to show that this inequality holds for $N=2$, {\sl viz.}\ 
\[
A_\text{FP}<\frac{R}{2\beta-\nu}.
\]
Given the properties of $F_{1}$ described in the previous subsection, in order to prove that inequality, it suffices to verify that $F_{1}(\frac{R}{2\beta-\nu})<\frac{R}{2\beta-\nu}$. Explicit computations show that the sign of $F_{1}( \frac{R}{2\beta-\nu})-\frac{R}{2\beta-\nu}$ is the same as the one of $\tanh\left(\frac{\beta^2R}{2\beta-\nu}\right)-\frac{\beta^2R}{2\beta-\nu}$, which is negative for every $R,\beta\in\R^+$ and $\nu\in (0,\beta)$. This concludes the proof of instability for $\tau=0$. \hfill $\Box$

\subsection{Proof of stability for $\tau>0$}\label{S-PSTABSYNC}
We first give some heuristic for the synchronization mechanism when $\tau>0$. Given a small initial perturbation $(\mathbf x,A)$ of the fixed point, the vector field in \eqref{ODE2}, when computed at instants immediately before the first firing (which occurs at time 0) must be negative because $m(t)$ must be close to $0$ and $A(t)$ must be close to $A_\text{FP}$, which is positive. Therefore for $\tau>0$ small enough and $x_i\in (0,\tau)$, since there cannot be any firing in the interval $[-\tau,x_i-\tau]$, we must have
\[
\phi_m^{x_i-\tau}(A)=\phi_m^{-\tau}(A)- C'x_i+\ \text{h.o.t.},
\]
for some $C'\in \R^+_\ast$, which implies that the image $x'_i$ of $x_i$ under $F_N$ is given by
\[
x'_i=x_i+\nu (\phi_m^{x_i-\tau}(A)-\phi_m^{-\tau}(A))=x_i(1-\nu C')+\ \text{h.o.t.}
\]
Thus, we must have $x'_i<x_i$ when $(\mathbf x,A)$ is sufficiently close to $(0,\cdots,0,A_\text{FP})$, which together with the fact that $x'_i>0$ (from order preservation), implies the desired synchronisation.
 
\paragraph{Proof of item {\sl (ii)} in Proposition \ref{MAINRES}} We are going to show that for every $N\geq 2$, the map $F_N$ is a contraction in a small neighbourhood of $(0,\cdots ,0,A_\text{FP})$ in $(\R^+)^N$ when $\tau>0$ is small enough. Given $\tau\in (0,R)$, let $(\mathbf x,A)$ with $x_N=0$, and $\|\mathbf x\|_{N-1},|A-A_\text{FP}|$ small. In particular, we assume that $\|\mathbf x\|_{N-1}<\tau$ and that $(\mathbf x,A)$ satisfies the condition \eqref{DOMAINFN}.

Let $i\in\{1,\cdots ,N-1\}$. The condition \eqref{DOMAINFN} ensures that no firing occurs in the time interval $[-\tau,0)$; hence 
\[
m(t)=m(0)-t\quad\text{for}\quad t\in [-\tau,x_i-\tau],
\]
from where we obtain using \eqref{Eq:EvforA}, after simple algebra
\[
\phi^{x_i-\tau}_m(A)=\left(\phi^{-\tau}_m(A)+\int_{0}^{x_i}e^{\beta s}m(s-\tau)ds\right)e^{-\beta x_i}=\phi^{-\tau}_m(A)-x_i\left(\beta \phi^{-\tau}_m(A)-\tau\right)+ {\cal O}(\|\mathbf x\|_{N-1}^2)
\]
and thus, the following expansion results for the coordinate $x'_i$ of the first iterate $(\mathbf x',A')=F_N(\mathbf x,A)$
\begin{align*}
x'_i&=x_i\left(1-\nu \left(\beta \phi^{-\tau}_m(A)-\tau\right)\right)+ {\cal O}(\|\mathbf x\|_{N-1}^2)\\
&=x_i(1-\nu \beta A_\text{FP}+(\beta(A-\phi^{-\tau}_m(A))-\tau))+\nu \beta x_i(A_\text{FP}-A)+ {\cal O}(\|\mathbf x\|_{N-1}^2)
\end{align*}
That no firing occurs in the interval $[-\tau,0)$ implies that $\phi^{-\tau}_m$ is close to the identity when $\tau$ is small. Therefore, there exists $K_\tau>0$ with $\lim_{\tau\to 0} K_\tau=0$ such that 
\[
|\beta(A-\phi^{-\tau}_m(A))-\tau|\leq K_\tau
\]
for all $A$ uniformly bounded, and in particular when $|A-A_\text{FP}|$ is small. In addition, order preservation, which must hold for the initial conditions $(\mathbf x,A)$ under consideration here, implies that $x'_i\geq 0$ for all $i$. Letting $\tau\to 0$ and $A\to A_\text{FP}$, we must have $1-\nu \beta A_\text{FP}\geq 0$. Altogether, this implies the existence of $\gamma_1\in (0,1)$ such that 
\[
\|\mathbf x'\|_{N-1}\leq \left(\gamma_1+\nu\beta |A-A_\text{FP}|\right)\|\mathbf x\|_{N-1},
\]
when $\tau,\|\mathbf x\|_{N-1}$ and $|A-A_\text{FP}|$ are sufficiently small.

In order to control the dynamics of the activator variable, we are going to show that $A'$ is close to $F_1(A)$ when $\|\mathbf x\|_{N-1}$ is small and then use that $F_1$ is a contraction in the neighbourhood of $A_\text{FP}$ in $\R$. Recalling the notation $t_\text{R}=R+\nu \phi_m^{-\tau}(A)$ for the return time and letting $t_\text{R}^\text{sync}=R+\nu \phi_{m_\text{sync}}^{-\tau}(A)$ and $m_\text{sync}$ for the quantities associated with the synchronized trajectory issued from $(0,A)$ at $t=0$, we observe that together with \eqref{BCKVC}, the expressions
\[
m(t)=\frac{1}{N}\sum_{i=1}^{N-1}x_i-t\quad \text{and}\quad m_\text{sync}(t)=-t\quad \text{for}\ t\in [-\tau,0],
\]
result in 
\[
t_\text{R}-t_\text{R}^\text{sync}=-\nu \frac{e^{\beta\tau}-1}{\beta N}\sum_{i=1}^{N-1}x_i<0.
\]
Accordingly, we obtain from \eqref{Eq:EvforA}
\begin{align*}
A'=&F_1(A)+\left(A+\int_0^{t_\text{R}}e^{\beta s}m_\text{sync}(s)ds\right)(e^{-\beta t_\text{R}}-e^{-\beta t_\text{R}^\text{sync}})+\int_0^{t_\text{R}}e^{\beta (s-t_\text{R})}(m(s)-m_\text{sync}(s))ds\\
&+\int_{t_\text{R}}^{t_\text{R}^\text{sync}}e^{\beta (s-t_\text{R}^\text{sync})}m_\text{sync}(s)ds
\end{align*}
The expression of $t_\text{R}-t_\text{R}^\text{sync}$ above implies that the second and last terms in the RHS are controlled by $\|\mathbf x\|_{N-1}$, also because all quantities $A,t_\text{R},t_\text{R}^\text{sync}$ and $m_\text{sync}$ are bounded when in (or close to) the attracting set. 

In order to control the third term, we first observe that it suffices to provide an estimate of the integral between time $t=\|x\|_{N-1}$ (namely the time of the last first firing of all oscillators) and $t=t_\text{R}$ because the remaining integral can be controlled by the same arguments as before. Moreover, we have 
\[
m(s)=R+\frac{\nu}{N}\sum_{i=1}^N\phi_m^{x_i-\tau}(A)+\frac{1}{N}\sum_{i=1}^{N-1}x_i-s\quad\text{and}\quad m_\text{sync}(s)=R+\nu \phi_{m_\text{sync}}^{-\tau}(A)-s\quad \text{for}\ s\in [\|x\|_{N-1},t_\text{R}]
\]
and again from \eqref{Eq:EvforA}
\[
|\phi_m^{x_i-\tau}(A)-\phi_m^{-\tau}(A)|\leq \phi_m^{-\tau}(A)(1-e^{-\beta x_i})+\|m|_{[-\tau,x_i-\tau]}\|_0\frac{e^{\beta x_i}-1}{\beta}.
\]
In addition, from the expressions of $m$ and $m_\text{sync}$ above, we have 
\[
\phi_m^{-\tau}(A)=\phi_{m_\text{sync}}^{-\tau}(A)-\frac{e^{\beta\tau}-1}{\beta N}\sum_{i=1}^{N-1}x_i.
\]
Combining all the estimates above, we finally conclude about the existence of $K\in\R^+$ such that
\[
|A'-A_\text{FP}|\leq |F_1(A)-A_\text{FP}|+K\|\mathbf x\|_{N-1}+ {\cal O}(\|\mathbf x\|_{N-1}^2)
\]
when $\|\mathbf x\|_{N-1}$ is sufficiently small. Besides, the analysis of the map $F_1$ in a previous section implies the existence of $\gamma_2\in (0,1)$ such that 
\[
|F_1(A)-A_\text{FP}|\leq \gamma_2|A-A_\text{FP}|
\]
when $A$ is sufficiently close to $A_\text{FP}$. 

Altogether, when $\tau,\|\mathbf x\|_{N-1}$ and $|A-A_\text{FP}|$ are sufficiently small, the return dynamics of $\|\mathbf x\|_{N-1}$ and $|A-A_\text{FP}|$ is dominated by a matrix whose eigenvalues are close to $\gamma_1$ and $\gamma_2$. Hence, the map $F_N$ must be a contraction (for an appropriate norm in $\R^N$) in the neighbourhood of $(0,\cdots, 0,A_\text{FP})$, as desired. In particular, we are sure that $(\mathbf x',A')$ also satisfies condition \eqref{DOMAINFN} provided that $\|\mathbf x\|_{N-1},|A-A_\text{FP}|$ are sufficiently small; hence the argument can be repeated to conclude about asymptotic stability of $(0,\cdots, 0,A_\text{FP})$. \hfill $\Box$

\section{Stability analysis for partially synchronized periodic orbits}\label{S-PARTSYNC}
The analysis in the previous section and its arguments are not limited to synchronized trajectories. They extend to partially synchronized periodic orbits. A {\bf partially synchronized trajectory} (whether it is periodic or not) is a solution of the equations \eqref{ODE1}-\eqref{ODE2} for which two or more repressor concentrations are equal at all times (NB: this is the case iff the concentrations are equal at $t=0$). 

For an arbitrary fixed point of the return map in a given partially synchronized subspace (see next section for an accurate definition), we establish a stability criterion for $\tau>0$ (Lemma \ref{STABCRIT}) and an instability criterion for $\tau=0$ (Lemma \ref{INSTABCRIT}). For simplicity, we only consider stability with respect to initial perturbations that affect a single cluster of the periodic orbit under investigation. The stability with respect to perturbations that affect several clusters are direct extensions that are left to the interested reader. In addition, we also consider initial conditions that are given by the datum $(\mathbf x, A|_{[-\tau,0]})$, even though we have argued that only finite many values of $A$ suffice in order to define any trajectory. 

In a second step, we apply these criteria to fixed points that are exponentially stable in their own partially synchrony subspace for $\tau=0$. The result for the continued fixed point for $\tau>0$ small can be regarded as an extension of Proposition \ref{MAINRES} to the family under consideration: while for $\tau=0$, the orbit is unstable with respect to perturbations that smear its clusters, it becomes stable against the same perturbations for every positive (and small) value of $\tau$ (Theorem \ref{STABPARTSYNC}). An example of application to fixed points with two clusters with equi-distributed repressor concentrations is given in Lemma \ref{TWOCLUST}.

\subsection{Characteristics of partially synchronized trajectories and of fixed points of the corresponding return map} 
By grouping the oscillators with equal repressor concentration into one cluster, the population at every instant can be described by the vector $\left(\{n_k,y_k\}_{k=1}^K,A\right)$ where $n_k\in\{1,\cdots ,N\}$ denotes the size of the cluster $k$ and $y_k$ the corresponding repressor concentration ($K\leq N$ is the total number of clusters and we have $\sum_{k=1}^Kn_k=N$).\footnote{The variables $y_k$ can be formally defined as follows. For each $i\in \{1,\cdots ,N\}$, there exists $k_i\in \{1,\cdots ,K\}$ such that $y_{k_i}=x_i$ and $x_{j}\neq x_i$ iff $y_{k_j}\neq y_{k_i}$.} From the equations \eqref{ODE1}-\eqref{ODE2}, it follows that the {\bf cluster distribution} $\{n_k\}_{k=1}^K$ remains constant in every trajectory; hence we may consider separately the dynamics in each subspace, called {\bf partially synchronized subspace}, for which this distribution is given. Every trajectory for which $K<N$ is called {\bf partially synchronized}. Obviously, the synchronized dynamics in Section \ref{S-SYNCHRO} is a particular case of partially synchronized subspace ($K=1$).

Let $N\in\N$, $K\in\{1,\cdots ,N-1\}$ and a cluster distribution $\{n_k\}_{k=1}^K$ be given. For the sake of the presentation, we only consider those partially synchronized periodic orbits  that are fixed points 
\[
\left(\{n_k,y^\text{FP}_k\}_{k=1}^{K},A^\text{FP}|_{[-\tau,0]}\right)
\]
of the return map to the Poincar\'e section $y_K=0$. In other words, we assume that the period $T_\text{FP}>0$ coincides with the return time. However the analysis developed below extends to arbitrary periodic orbits of such map, without additional conceptual difficulties.  

Notice also that the labelling of the clusters is irrelevant because of the permutation symmetry. For convenience, we assume that this labelling has been chosen so that the fixed point repressor concentrations $y_k^\text{FP}$ are ordered {\sl ie.}\
\[
y_1^\text{FP}>y_2^\text{FP}>\cdots >y_{K-1}^\text{FP}>0.
\]
Furthermore, the expression \eqref{Eq:EvforA} and the fact that the non-negative function $m^\text{FP}$ certainly does not entirely vanish over $[0,T_\text{FP}]$, impose that, in any partially synchronized periodic orbit, the activator function $A^\text{FP}$ must be positive.

\subsection{Considerations about stability} 
As technical considerations about stability are concerned, focus will be made on the strongest form of local Lyapunov stability, namely the exponential stability that results when the return dynamics of small perturbations is a contraction in an appropriate setting (space and norm). In particular, a prerequisite for our stability criterion will be that the periodic orbit is exponentially stable inside its proper partially synchronized space. The criterion will then ensure exponential stability in a higher dimensional space that contains perturbations that smear the clusters. More precisely, we shall deal with the following notions (NB: Throughout, the symbol $\mathbf y$ denotes the collection $\{y_k\}_{k=1}^{K}$).
\begin{Def}
(i) The fixed point $\left(\{n_k,y^\text{FP}_k\}_{k=1}^{K},A^\text{FP}|_{[-\tau,0]}\right)$ is said to be {\bf exponentially stable inside its proper partially synchronized subspace} if there exist $\epsilon_0\in \R_\ast^+$ and $\gamma\in (0,1)$ such that for every $\epsilon\in (0,\epsilon_0)$ and every initial datum $\left(\{n_k,y_k\}_{k=1}^{K},A|_{[-\tau,0]}\right)$ such that
\[
y_{K}=0\quad\text{and}\quad\max\left\{\|\mathbf y-\mathbf y^\text{FP}\|_{K-1},\|A|_{[-\tau,0]}-A^\text{FP}|_{[-\tau,0]}\|_0\right\}\leq \epsilon,
\]
we have for the subsequent trajectory $t\mapsto \left(\{n_k,y_k(t)\}_{k=1}^{K},A(t)\right)$ of the system \eqref{ODE1}-\eqref{ODE2}
\[
\max\left\{\|\mathbf y(t_\text{R})-\mathbf y^\text{FP}\|_{K-1},\|A|_{[t_\text{R}-\tau,t_\text{R}]}-A^\text{FP}|_{[-\tau,0]}\|_0\right\}\leq \gamma \epsilon,
\]
where the return time $t_\text{R}$ defined in Section \ref{S-RETMAP} corresponds here to the instant of the second firing of the oscillators in cluster $K$.
\medskip

\noindent
(ii) The fixed point $\left(\{n_k,y^\text{FP}_k\}_{k=1}^{K},A^\text{FP}|_{[-\tau,0]}\right)$ is said to be {\bf exponentially stable with respect to small perturbations that smear cluster $K$} if there exist $\epsilon_0,C_1,C_2\in \R_\ast^+$ and $\gamma\in (0,1)$ such that for every $\epsilon\in (0,\epsilon_0)$, every $K'\in \{K+1,\cdots ,N\}$, every cluster distribution $\{n'_k\}_{k=1}^{K'}$ such that $n'_k=n_k$ for $k\in\{1,\cdots ,K-1\}$, and every initial datum $\left(\{n'_k,y_k\}_{k=1}^{K'},A|_{[-\tau,0]}\right)$ such that $y_{K'}=0$ and  
\[
\max\left\{\|\mathbf y-\mathbf y^\text{FP}\|_{K-1},C_1\max_{k\in \{K,\cdots ,K'-1\}}y_k,\|A|_{[-\tau,0]}-A^\text{FP}|_{[-\tau,0]}\|_0\right\}\leq \epsilon,
\]
we have for the subsequent trajectory $t\mapsto \left(\{n'_k,y_k(t)\}_{k=1}^{K'},A(t)\right)$ of the system \eqref{ODE1}-\eqref{ODE2}
\[
\max\left\{\|\mathbf y(t^n_\text{R})-\mathbf y^\text{FP}\|_{K-1},C_1\max_{k\in \{K,\cdots ,K'-1\}}y_k(t^n_\text{R}),\|A|_{[t^n_\text{R}-\tau,t^n_\text{R}]}-A^\text{FP}|_{[-\tau,0]}\|_0\right\}\leq C_2\gamma^n \epsilon,\ \forall n\in\N
\]
where the $n$th return time $t^n_\text{R}$ is the instant of the $(n+1)$th firing of the oscillators in cluster $K'$.
\end{Def}
Anticipating the comment after Lemma \ref{STABCRIT}, notice that exponential stability inside the proper partially synchronized subspace does not depend on the phase of the periodic orbit under consideration, {\sl viz.}\ the fixed point $\left(\{n_k,y^\text{FP}_k\}_{k=1}^{K},A^\text{FP}|_{[-\tau,0]}\right)$ is exponentially stable inside its proper partially synchronized subspace iff the fixed point 
\[
\left(\{n'_k,y'_k\}_{k=1}^{K},A^\text{FP}|_{[y^\text{FP}_{K-1}-\tau,y^\text{FP}_{K-1}]}\right),
\]
where  
\[
n'_k=\left\{\begin{array}{ccl}
n_K&\text{if}&k=1\\
n_{k-1}&\text{if}&k\in\{2,\cdots,K\}
\end{array}\right.\quad\text{and}\quad
y'_k=\left\{\begin{array}{ccl}
R+\nu A(y^\text{FP}_{K-1}-\tau)&\text{if}&k=1\\
y^\text{FP}_{k-1}-y^\text{FP}_{K-1}&\text{if}&k\in\{2,\cdots,K\}
\end{array}\right.,
\]
is also exponentially stable in the same sense. This is a standard consequence of the continuity of the firing dynamics ({\sl ie.}\ the fact that the coordinates immediately after firing depend continuously on the coordinate immediately after the previous firing).

\subsection{Stability criterion for $\tau>0$}
The stability criterion for partially synchronized periodic orbits, which is given in the next statement, is reminiscent of the heuristic argument for the stability of the fully synchronized periodic orbit given at the beginning of section \ref{S-PSTABSYNC}.
\begin{Lem}
Assume that \eqref{CONDORDER} holds. Given $N\in\N$, $K\in\{1,\cdots ,N-1\}$, a cluster distribution $\{n_k\}_{k=1}^K$ and $\tau>0$, assume that the return map in the Poincar\'e section $y_K=0$ in the corresponding partially synchronized subspace has a fixed point $\left(\{n_k,y^\text{FP}_k\}_{k=1}^{K},A^\text{FP}|_{[-\tau,0]}\right)$. 
Assume also that the following conditions hold
\begin{itemize}
\item the fixed point is asymptotically stable inside its proper partially synchronized subspace,
\item the periodic orbit of \eqref{ODE1}-\eqref{ODE2} that passes through the fixed point has no firing in any of the time intervals $[-\tau,0]\ \text{mod}\ T_\text{FP}$, and the derivative of $A$ at $t=-\tau$ is negative, {\sl ie.}\ $\dot A^\text{FP}(-\tau)=m_\text{FP}(-\tau)-\beta A^\text{FP}(-\tau)<0$.  
\end{itemize}
Then, the fixed point is exponentially stable with respect to small perturbations that smear cluster $K$.
\label{STABCRIT}
\end{Lem}
That the second condition in this statement depends on the value of the derivative $\dot A$ at a certain time suggests that stability with respect to cluster smearing {\sl a priori} depends on the cluster under consideration. However, we shall see in the proof of Theorem \ref{STABPARTSYNC} below that, provided that $\tau$ is sufficiently small, this is not the case because the sign of the derivative involved actually does not depend on the cluster under consideration.

Notice also that in the special case of full synchrony ($K=1)$, the statement is relevant to the unique synchronized periodic orbit of Section \ref{S-SYNCHRO}. Recall that this orbit exponentially attracts all synchronized trajectories. Moreover, the condition $\tau<R$ ensures that no firing can happen in the interval $[-\tau,0]$. In addition, since $m_\text{FP}(t)=-t$ for $t\in [-\tau,0]$, we have $\dot A^\text{FP}(-\tau)=\tau-\beta A^\text{FP}(-\tau)$. Since we showed that $A^\text{FP}(0)>0$, the continuity at 0 of the map $\tau\mapsto \tau-\beta A^\text{FP}(-\tau)$ implies that the conditions of Lemma \ref{STABCRIT} certainly hold when $\tau$ is small enough, {\sl viz.}\ statement {\sl (ii)} of Proposition \ref{MAINRES} is recovered as a particular instance where Lemma \ref{STABCRIT} applies.
\medskip

\noindent
{\sl Proof of Lemma \ref{STABCRIT}.} 
Given $K'\in \{K+1,\cdots ,N\}$, let $\left(\{n'_k,y_k\}_{k=1}^{K'},A|_{[-\tau,0]}\right)$ with $y_{K'}=0$ be such that 
\[
\max\left\{\|\mathbf y-\mathbf y^\text{FP}\|_{K-1},C_1\max_{k\in \{K,\cdots ,K'-1\}}y_k,\|A|_{[-\tau,0]}-A^\text{FP}_{[-\tau,0]}\|_0\right\}\leq \epsilon
\]
where $\epsilon\in (0,\epsilon_0)$ is arbitrary and $\epsilon_0,C_1>0$ are to be determined. We are going to evaluate the amplitude of the perturbation at the instants $t^n_\text{R}$  by considering separately the coordinates $y_k(t^n_R)$ for $k\in\{1,\cdots K-1\}$ to those for $k\in \{K,\cdots ,K'-1\}$.
\medskip

\noindent
{\sl Analysis of the coordinates $\{y_k(t^n_\text{R})\}_{k=1}^{K-1}$ and $A|_{[t^n_\text{R}-\tau,t^n_\text{R}]}$.} For these coordinates, we use the fact that the fixed point is asymptotically stable inside its proper partially synchronized subspace. To do so, we rely on a property of Lipschitz continuity of the return map, which is claimed and proved in Appendix \ref{A-LIPCONT}.

Let $t\mapsto \left(\{n_k,y^\text{p.sync}_k(t)\}_{k=1}^K,A^\text{p.sync}(t)\right)$ be the partially synchronized solution issued from the initial datum $\left(\{n_k,(1-\delta_{k,K})y_k\}_{k=1}^K,A|_{[-\tau,0]}\right)$.\footnote{$\delta_{i,j}$ is the Kronecker symbol.} Notice that the value of $A(-\tau)$ in this trajectory is the same as the value of $A(-\tau)$ in the original trajectory $t\mapsto \left(\{n'_k,y_k(t)\}_{k=1}^{K'},A(t)\right)$. Hence, the first return time to $y_K=0$ for this trajectory is the same as the first return time $t^1_\text{R}$ of the original trajectory. 

Consider the decomposition
\[
\|\mathbf y(t^1_\text{R})-\mathbf y^\text{FP}\|_{K-1}\leq \|\mathbf y(t^1_\text{R})-\mathbf y^\text{p.sync}(t^1_\text{R})\|_{K-1}+\|\mathbf y^\text{p.sync}(t^1_\text{R})-\mathbf y^\text{FP}\|_{K-1}.
\]
By periodicity of the trajectory passing through $\left(\{n_k,y^\text{FP}_k\}_{k=1}^{K},A^\text{FP}|_{[-\tau,0]}\right)$, we must have 
\[
y_1^\text{FP}=\max_{k\in\{1,\cdots ,K-1\}}y^\text{FP}_k<T_\text{FP}=R+\nu A^\text{FP}(-\tau).
\]
Hence, provided that $\epsilon$ is sufficiently small, we have 
\[
y_1=\max_{k\in\{1,\cdots ,K'-1\}}y_k<t^1_\text{R}=R+\nu A(-\tau)
\]
Applying Lemma \ref{LIPCONT} in Appendix \ref{A-LIPCONT}, we obtain the following inequality, when regarding $\mathbf y^\text{p.sync}$ and $\mathbf y^\text{p.sync}(t^1_\text{R})$ as elements of the partially synchronized subspace defined by $\{n'_k\}_{k=1}^{K'}$ whose coordinates $k\in \{K,\cdots ,K'\}$ all vanish
\[
\|\mathbf y(t^1_\text{R})-\mathbf y^\text{p.sync}(t^1_\text{R})\|_{K-1}\leq \|\mathbf y(t^1_\text{R})-\mathbf y^\text{p.sync}(t^1_\text{R})\|_{K'-1}\leq L\|\mathbf y-\mathbf y^\text{p.sync}\|_{K'-1}=L\max_{k\in\{K,\cdots ,K'-1\}}y_k\leq L\frac{\epsilon}{C_1}.
\]
On the other hand, asymptotic stability inside the partially synchronized subspace implies the existence of $\gamma_1\in (0,1)$ such that, provided that $\epsilon$ is sufficiently small, we have 
\[
\|\mathbf y^\text{p.sync}(t^1_\text{R})-\mathbf y^\text{FP}\|_{K-1}\leq \gamma_1\epsilon,
\]
and then
\[
\|\mathbf y(t^1_\text{R})-\mathbf y^\text{FP}\|_{K-1}\leq \left(\gamma_1+\frac{L}{C_1}\right)\epsilon.
\]
A similar reasoning applies to $\|A|_{[t^1_\text{R}-\tau,t^1_\text{R}]}-A^\text{FP}_{[-\tau,0]}\|_0$, which yields the same inequality. As a consequence, provided that $C_1$ is large enough so that $\gamma_2:=\gamma_1+\frac{L}{C_1}<1$, we get 
\[
\max\left\{\|\mathbf y(t^1_\text{R})-\mathbf y^\text{FP}\|_{K-1},\|A|_{[t^1_\text{R}-\tau,t^1_\text{R}]}-A^\text{FP}_{[-\tau,0]}\|_0\right\}\leq \gamma_2 \epsilon.
\]
By induction, the same arguments show that for every sufficiently small $\epsilon\in\R^+$, we have 
\[
\max\left\{\|\mathbf y(t^n_\text{R})-\mathbf y^\text{FP}\|_{K-1},\|A|_{[t^n_\text{R}-\tau,t^n_\text{R}]}-A^\text{FP}_{[-\tau,0]}\|_0\right\}\leq \gamma_2^n \epsilon,\ \forall n\in\N
\]
provided that one can simultaneously ensure 
\[
\max_{k\in\{K,\cdots ,K'-1\}}y_k(t^n_\text{R})\leq \frac{\gamma_2^n\epsilon}{C_1}.
\]
\medskip

\noindent
{\sl Analysis of the coordinates $\{y_k(t^n_\text{R})\}_{k=K}^{K'-1}$.} For these coordinates, we separate the cases $n=1$ and $n>1$. For $n=1$, we simply refer to Lemma \ref{LIPCONT} to obtain 
\[
\max_{k\in \{K,\cdots ,K'-1\}}y_k(t^1_\text{R})\leq L\epsilon,
\]
and the desired inequality holds for any pair $\gamma,C_2$ such that $L\leq \frac{C_2 \gamma}{C_1}$. For $n>1$, we proceed by induction. From the computation of the return map in Section \ref{S-RETMAP}, we have in particular at the instant $t^2_\text{R}$ of the third firing of the cluster $K'$
\[
y_k(t^2_\text{R})=y_k(t^1_\text{R})+\nu \left(A(y_k(t^1_\text{R})+t^1_\text{R}-\tau)-A(t^1_\text{R}-\tau)\right),\ \forall k\in \{K,\cdots ,K'-1\}.
\]
The assumption that the periodic orbit does not fire in the time interval $[T_\text{FP}-\tau,T_\text{FP}]$ is equivalent to assuming that $T_\text{FP}-y_1^\text{FP}>\tau$ (NB: recall that $y_1^\text{FP}=\max_k y_k^\text{FP}$). Using that $t^1_\text{R}-T_\text{FP}=\nu (A(-\tau)-A^\text{PF}(-\tau))$, the assumption on the initial datum $\left(\{n'_k,y_k\}_{k=1}^{K'},A|_{[-\tau,0]}\right)$ implies that we also have $t^1_\text{R}-y_1>\tau$ (and $y_1=\max_k y_k$) when $\epsilon$ is sufficiently small, {\sl viz.}\ the subsequent trajectory does not fire in the time interval $[t^1_\text{R}-\tau,t^1_\text{R}]$. Therefore, for $\epsilon$ sufficiently small, we have
\begin{align*}
y_k(t^2_\text{R})&=y_k(t^1_\text{R})\left(1+\nu \dot A(t^1_\text{R}-\tau)\right)+{\cal O}\left((y_k(t^1_\text{R}))^2\right)\\
&=y_k(t^1_\text{R})\left(1+\nu (m(t^1_\text{R}-\tau)-\beta A(t^1_\text{R}-\tau))\right)+{\cal O}\left((y_k(t^1_\text{R}))^2\right),\ \forall k\in \{K,\cdots ,K'-1\}
\end{align*}
The analysis above of $\|A|_{[t^1_\text{R}-\tau,t^1_\text{R}]}-A^\text{FP}_{[-\tau,0]}\|_0$ showed that $|A(t^1_\text{R}-\tau)-A^\text{FP}(-\tau)|$ can be made arbitrarily small by taking $\epsilon$ sufficiently small. Moreover, the fact that no firing takes place in the time interval $[t^1_\text{R}-\tau,t^1_\text{R}]$ implies that
\[
m(t^1_\text{R}-\tau)=m(t^1_\text{R})+\tau.
\]
Similarly, we have $m_\text{FP}(-\tau)=m_\text{FP}(0)+\tau$ for the periodic trajectory. In addition, the arguments above showed that $\|\mathbf y(t^1_\text{R})-\mathbf y^\text{FP}\|_{K'-1}$ can be made arbitrarily small by taking $\epsilon$ sufficiently small.  Consequently, the same property holds for $|m(t^1_\text{R}-\tau)-m_\text{FP}(-\tau)|$. Altogether, using the assumption $m_\text{FP}(-\tau)-\beta A^\text{FP}(-\tau)<0$, this proves the existence of $\gamma_3\in (0,1)$ such that, provided that $\epsilon$ is sufficiently small, we have
\[
y_k(t^2_\text{R})\leq \gamma_3 y_k(t^1_\text{R}),\ \forall k\in \{K,\cdots ,K'-1\},
\]
from where the desired induction follows with $\gamma=\max\{\gamma_2,\gamma_3\}$ when the results of the analysis of the other coordinates are also taken into account. The proof of the Lemma is complete. \hfill $\Box$

\subsection{Instability criterion for $\tau =0$}
The announced instability criterion for partially synchronized periodic orbits is given in the next statement. 
\begin{Lem}
Assume that \eqref{CONDORDER} holds. Given $N\in\N$, $K\in\{1,\cdots ,N-1\}$ and a cluster distribution $\{n_k\}_{k=1}^K$, assume that, for $\tau=0$, the return map in the Poincar\'e section $y_K=0$ in the corresponding partially synchronized subspace has a fixed point $\left(\{n_k,y^\text{FP}_k\}_{k=1}^{K},A^\text{FP}\right)$. Let $\dot A^\text{FP}(0^+)$ be the right limit at 0 of the derivative of the activator concentration in the corresponding trajectory. Then, under the condition
\[
 \dot A^\text{FP}(0^+)>\frac{R+\nu A^\text{FP}}{N}
\]
the fixed point is unstable with respect to some arbitrarily small perturbations of cluster $K$.
\label{INSTABCRIT}
\end{Lem}
In the case of the fully synchronized periodic orbit ($K=1$), we have $\dot A^\text{FP}(0^+)=R+\nu A^\text{FP}-\beta A^\text{FP}$; hence the condition in the Proposition becomes
\[
\frac{N-1}N(R+\nu A^\text{FP})-\beta A^\text{FP}>0,
\]
which is exactly the identified and proved criterion in the proof of statement {\em (i)} in Proposition \ref{MAINRES}.
\medskip

\noindent
{\sl Proof of Lemma \ref{INSTABCRIT}.} Given $K'\in \{K+1,\cdots ,N\}$, let $\left(\{n'_k,y_k\}_{k=1}^{K'},A(0)\right)$ with $y_{K'}=0$ be an initial condition such that
\[
y_k=\left\{\begin{array}{ccl}
y_k^\text{FP}&\text{if}&k\in\{1,\cdots, K-1\}\\
y_K&\text{if}&k=K\\
0&\text{if}&k\in\{ K+1,\cdots ,K'\}
\end{array}\right.
\]
where $y_K>0$ is sufficiently small. There is no firing in the time interval $(0,y_K)$. Hence, we have
\[
y_K(t_\text{R})=y_K+\nu(A(y_K)-A(0))=y_K(1+\nu \dot A(0^+))+{\cal O}(y_K^2)
\]
Moreover, equation \eqref{ODE2} implies that $\dot A(0^+)=m(0^+)-\beta A(0)$. Using
\[
y_k(0^+)=\left\{\begin{array}{ccl}
y_k^\text{FP}&\text{if}&k\in\{1,\cdots, K-1\}\\
y_K&\text{if}&k=K\\
R+\nu A(0)&\text{if}&k\in\{ K+1,\cdots ,K'\}
\end{array}\right.\quad \text{and}\quad 
y_k^\text{FP}(0^+)=\left\{\begin{array}{ccl}
y_k^\text{FP}&\text{if}&k\in\{1,\cdots, K-1\}\\
R+\nu A^\text{FP}&\text{if}&k=K
\end{array}\right.
\]
we get
\[
m(0^+)=m^\text{FP}(0^+)-\frac{R+\nu A(0)}{N}+y_K+\frac{n_K\nu (A(0)-A^\text{FP})}{N}.
\]
It follows that $m(0^+)-\beta A(0)$ can be made arbitrarily close to $\dot A^\text{FP}(0^+)-\frac{R+\nu A^\text{FP}}{N}$ when $y_K$ and $|A(0)-A^\text{FP}|$ are sufficiently small. The Lemma easily follows. 
\hfill $\Box$

\subsection{Delay-dependent stability in arbitrary large populations}
Let $K\in\N$ and $\{n_k\}_{k=1}^K\in\N^K$ be given. The equations \eqref{ODE1}-\eqref{ODE2} of the dynamics imply that if $t\mapsto \left(\{n_k,y_k(t)\}_{k=1}^{K},A(t)\right)$ is a partially synchronized trajectory in the population of $\sum_{k=1}^Kn_k$ oscillators, then for every $q\in\N$, the function $t\mapsto \left(\{q n_k,y_k(t)\}_{k=1}^{K},A(t)\right)$ is a partially synchronized trajectory in the population of $q\sum_{k=1}^Kn_k$ oscillators, {\sl viz.}\ the existence and the coordinates of partially synchronized trajectories do not depend on the (common) scaling of its cluster sizes.

The results on the stability analysis of partially synchronized periodic orbits are summarized in the following statement, whose conclusions can be regarded as some extension of Proposition \ref{MAINRES} to  fixed points of the return map asssociated with an arbitrary cluster distribution (with given relative cluster sizes).  
\begin{Thm}
Assume that \eqref{CONDORDER} holds. Given $K\in\N$ and $\{n_k\}_{k=1}^K\in\N^K$, assume that for $\tau=0$, the return map  
in the partially synchronized subspace associated with $\{n_k\}_{k=1}^K$, of the dynamics of a population of $\sum_{k=1}^Kn_k$ DF oscillators, has an exponentially stable fixed point, say $\left(\{n_k,y^\text{FP}_k\}_{k=1}^{K},A^\text{FP}\right)$. Then, the following assertions hold.

\noindent
(i) There exists $q_0\in\N$ such that for $\tau=0$ and every $q>q_0$, the fixed point $\left(\{q n_k,y^\text{FP}_k\}_{k=1}^{K},A^\text{FP}\right)$ is unstable with respect to some arbitrarily small perturbations of any of its clusters.

\noindent
(ii) There exists $\tau_0\in \R_\ast^+$ such that for every $\tau\in (0,\tau_0)$ and $q\in \N$, the return map of the partially synchronized dynamics with delay $\tau$ has a fixed point which is the continuation of $\left(\{q n_k,y^\text{FP}_k\}_{k=1}^{K},A^\text{FP}\right)$. This continued fixed point is exponentially stable with respect to small perturbations that smear any of its clusters.
\label{STABPARTSYNC}
\end{Thm}
As an application, we provide in the next section, an example of periodic orbits that fit the condition of this statement for $\beta$ sufficiently large, namely the periodic orbits with equi-distributed repressor concentrations. Evidently, the fully synchronized fixed point also satisfies the assumption of Theorem \ref{STABPARTSYNC}, as already claimed at the begining of Proposition \ref{MAINRES}.\footnote{Notice that item {\sl (i)} in Theorem \ref{STABPARTSYNC} is slightly weaker than the corresponding item in Proposition \ref{MAINRES} (which claims that $q_0=1$ in this case).} 
\medskip

\noindent
{\sl Proof of Theorem \ref{STABPARTSYNC}.} The proof decomposes into two parts. The first part establishes the existence of the continued fixed point for $\tau>0$ sufficiently small. The second part shows that such return map fixed points must appropriately satisfy the conditions of Lemmas \ref{STABCRIT} and \ref{INSTABCRIT}.  

For the first part, one preliminary shows that continuity arguments ensure that the conditions of Claim \ref{CLAIMFINITEDIM} hold for $\tau$ sufficiently small and in a sufficiently small neighbourhood of $\left(\{y^\text{FP}_k\}_{k=1}^{K},A^\text{FP}\right)$ in $(\R^+)^K$; hence the return dynamics in that neighbourhood reduces to a mapping of $(\R^+)^K$ into itself. Moreover, this map can be obtained as the composition of the maps, each called a {\bf firing map}, that bring the system from the state immediately before the firing of the $(k+1)$th cluster to the one of the $k$th cluster. The firing maps write, given a datum $\left(\{y_k\}_{k=1}^{K},A\right)$ for which $y_{k+1}=0$ and $y_k=\min_{\ell\neq k+1} y_\ell$
\[
y'_\ell=\left\{\begin{array}{ccl}
y_\ell-y_k&\text{if}&\ell\neq k+1\\
R+\nu \phi_m^{-\tau}(A)-y_k&\text{if}&\ell=k
\end{array}\right.\quad\text{and}\quad A'=\phi_m^{y_k}(A),
\]
where the expressions \eqref{Eq:EvforA} and \eqref{BCKVC} are to be used with respectively
\[
m(t)=\frac1{N}\left(\sum_{k=1}^Kn_ky_k+R+\nu \phi_m^{-\tau}(A)\right)-t,\ \forall t\in (0,y_k)\quad \text{and}\quad m(t)=\frac1{N}\sum_{k=1}^Kn_ky_k-t\ \forall t\in [-\tau,0].
\]
Accordingly, each firing map is differentiable in $\R^K$ and, together with its derivative, it continuously depends on $\tau$. Hence, so does the composed return map. The assumption that $\left(\{n_k,y^\text{FP}_k\}_{k=1}^{K},A^\text{FP}\right)$ is an exponentially stable fixed point implies, using the Implicit Function Theorem, that it can be uniquely continued as an exponentially stable fixed point of the return map for $\tau\in \R_\ast^+$ sufficiently small. The first part of the proof is complete.

The main argument of the second part consists in establishing the following signs for the left and right limits of the derivative of $A_\tau^\text{FP}$ (where the explicit dependence on $\tau$ has been added for clarity) at each firing
\[
\dot A_\tau^\text{FP}(\left(y_{\tau,k}^\text{FP}\right)^-)<0\quad\text{and}\quad \dot A_\tau^\text{FP}(\left(y_{\tau,k}^\text{FP}\right)^+)>0\ \forall k\in\{1,\cdots ,K\}.
\]
In order to prove these signs, notice that equation \eqref{ODE2} and the fact that $\dot m(t)=-1$ between two firings imply that if it happens that $\dot A(t)=0$ for some $t$ between two firings, then this can happen only once and $\dot A(t)>0$ for all $t$ before (resp.\ $\dot A(t)<0$ after) that instant. Moreover, equation \eqref{Eq:ReturnMap}, or the expression of $x_i(t_\text{R})$ in Section \ref{S-RETMAP} implies that, for a fixed point $\left(\{n_k,y^\text{FP}_{0,k}\}_{k=1}^{K},A_0^\text{FP}\right)$ of the return map for $\tau=0$, the activator concentration must be the same at each firing, {\sl ie.}\ 
\[
A_0^\text{FP}(y_{0,k}^\text{FP})=A_0^\text{FP},\ \forall  k\in\{1,\cdots ,K-1\}.
\]
Therefore, we must have $\dot A_0^\text{FP}(t)=0$ for some $t$ between any two firings, and then the desired signs hold for $\tau=0$. By continuous dependence on $\tau$ of the coordinates of the continued fixed point, the same signs hold for $A_\tau^\text{FP}$ provided that $\tau$ is small enough. 

Together with $\dot A_\tau^\text{FP}(\left(y_{\tau,k}^\text{FP}\right)^-)<0$, the implicit assumption in the continuation argument that there is no firing in any of the intervals $[y_{\tau,k}^\text{FP}-\tau,y_{\tau,k}^\text{FP}]$, and the fact that the fixed point coordinates do not depend on $q$ imply that the assumptions of Lemma \ref{STABCRIT} hold for the fixed point $\left(\{qn_k,y^\text{FP}_{0,k}\}_{k=1}^{K},A_0^\text{FP}\right)$, for every $q\in\N$. Moreover, the same assumptions hold for the return map fixed point of the corresponding periodic orbit in every Poincar\'e section $y_k=0$, implying stability with respect to smearing of any cluster.

Together with $\dot A_\tau^\text{FP}(\left(y_{\tau,k}^\text{FP}\right)^+)>0$ and the fact that the fixed point coordinates do not depend on $q$, we certainly have 
\[
\dot A_\tau^\text{FP}(\left(y_{\tau,k}^\text{FP}\right)^+)>\frac{R+\nu A^\text{FP}}{q\sum_{k=1}^Kn_k},
\]
when $q$ is sufficiently large. In this case, the assumptions of Lemma \ref{INSTABCRIT} hold, as they do for each corresponding return map fixed point in every Poincar\'e section $y_k=0$. The proof of Theorem is complete.
\hfill $\Box$

\subsection{Application to periodic orbits with equi-distributed repressor concentrations}\label{S-EQUIDIST}
As a example of application of Theorem \ref{STABPARTSYNC}, we consider in this section, return map fixed points for which the repressor coordinates are {\bf equi-distributed} among $N$ clusters, more precisely, those elements $(\mathbf x^\text{FP},A^\text{FP})\in \R^{N+1}$ such that $x^\text{FP}_N=0$, the function $m(s)=m^\text{FP}+s$ for $s\in [-\tau,0]$ and \footnote{By letting $x^\text{FP}_0=R+\nu \phi_m^{-\tau}(A^\text{FP})$, one could include the synchronized fixed point in this family, that would be obtained for $N=1$.} 
\[
x^\text{FP}_i-x^\text{FP}_{i+1}=\frac{R+\nu \phi_m^{-\tau}(A^\text{FP})}{N},\ \forall i\in\{1,\cdots ,N-1\}\quad\text{and}\quad A^\text{FP}=\phi_m^{x^\text{FP}_{N-1}}(A^\text{FP}).
\]
It is immediate to verifiy that the coordinates of $(\mathbf x^\text{FP},A^\text{FP})$ are recovered after every firing in the trajectory, not only under the action of the return map. Moreover, the repressor coordinates are entirely determined by $A^\text{FP}$. For $\tau=0$, an equation for $A^\text{FP}$  can be obtained, which can be shown to have exactly one solution, {\sl viz.}\ for each $N>1$, there exists a unique $N$-cluster equi-distributed fixed point. Moreover, a systematic stability analysis can be achieved for $\tau=0$, which is rigorous for $N\in \{2,\cdots ,5\}$ and concludes that the fixed point is stable in its partially synchrony subspace, provided that $\beta$ is sufficiently large. All these results are presented in the next statement.
\begin{Lem}
(i) For every $N\in\N$, $N>1$ and $\tau=0$, there exists a unique fixed point with equi-distributed repressor coordinates.

\noindent
(ii) Given $R,\nu$ and $N\in\{2,\cdots, 5\}$, there exists $\beta_N>0$ such that for every $\beta>\beta_N$, the equi-distributed fixed point is exponentially stable in its proper partially synchronized subspace.
\label{TWOCLUST}
\end{Lem}
We believe that statement {\em (ii)} holds for every $N>1$.  In order to ensure this property, it suffices to prove that the expression of the matrix $W_N$ defined in the proof below, holds for every $N$. Independently, by combining Lemma \ref{TWOCLUST} with Theorem \ref{STABPARTSYNC}, one immediately obtains the following conclusion.
\begin{Cor}
Given $N\in\{2,\cdots, 5\}$, assume that \eqref{CONDORDER} holds with $\beta>\beta_N$ so that for $\tau=0$, the equi-distributed fixed point with $N$ clusters is exponentially stable in its proper partially synchronized subspace.

\noindent
(i) There exists $q_N\in\N$ such that for $\tau=0$ and every $q>q_N$, the fixed point $\left(\{q ,x^\text{FP}_i\}_{i=1}^{N},A^\text{FP}\right)$ is unstable with respect to some arbitrarily small perturbations of any of its clusters.

\noindent
(ii) There exists $\tau_N\in \R_\ast^+$ such that for every $\tau\in (0,\tau_N)$ and $q\in \N$, the return map of the partially synchronized dynamics with delay $\tau$ has a fixed point which is the continuation of $\left(\{q,x^\text{FP}_i\}_{i=1}^{N},A^\text{FP}\right)$. This continued fixed point is exponentially stable with respect to small perturbations that smear any of its clusters.
\end{Cor}
One can show that the continuation of $(\mathbf x^\text{FP},A^\text{FP})$ for $\tau>0$ has indeed equi-distributed repressor concentrations. Moreover, recall from Claim \ref{CLAIMFINITEDIM} that $\tau_N$ also depends on the distance between the repressor concentrations immediately after firings. Hence, even if we assumed that statement {\em (ii)} in Lemma  \ref{TWOCLUST} held for all $N>1$ with $\sup_{N>1}\beta_N<+\infty$, we would not be able to ensure that $\inf_{N>1}\tau_N>0$. In order words, we do not know whether or not all equidistributed fixed points can be simultanously stable for some given delay $\tau>0$.
\medskip

\noindent
{\sl Proof of Lemma \ref{TWOCLUST}.}
{\em (i)} As usual, the return map fixed point $(\mathbf x^\text{FP},A^\text{FP})$ is assumed to coincide with the state at $t=0$ of the periodic orbit $t\mapsto (\mathbf x^\text{FP}(t),A^\text{FP}(t))$ of the continuous time system. The fixed point definition implies that there is no firing in the time interval $(0,x^\text{FP}_{N-1}]$ where $x^\text{FP}_{N-1}=\frac{R+\nu A^\text{FP}}{N}$. Hence, we have 
\[
m^\text{FP}(t)=m^\text{FP}(0^+)-t,\ \forall t\in (0,x_{N-1}]\ \text{where}\ m^\text{FP}(0^+)=\frac1{N}\left(\sum_{i=1}^{N-1}x_i^\text{FP}+R+\nu A^\text{FP}\right)=\frac{(N+1)(R+\nu A^\text{FP})}{2N}.
\]
Using \eqref{Eq:EvforA}, the fixed point equation $A^\text{FP}=\phi_{m^\text{FP}}^{x^\text{FP}_{N-1}}(A^\text{FP})$ then rewrites as $f(A^\text{FP})=0$ where $f$ is given by the following expression
 \[
f(A)=\left(\frac{(N+1)(R+\nu A)}{2 N}-\beta A+\frac1{\beta}\right)(1-e^{-\beta \frac{R+\nu A}{N}})-\frac{R+\nu A}{N}.
 \]
In particular, we have $f(0)=\frac{g(\beta R)}{\beta N}$ where 
 \[
 g(x)=\left(\frac{(N+1)x}{2 }+N\right)(1-e^{-\frac{x}{N}})-x.
 \]
 Basic calculations yield
 \[
 g(0)=0\quad \text{and}\quad g'(x)>0,\ \forall x>0,
 \]
hence $f(A)>0$ for all $A>0$ sufficiently small. Moreover, recall the upper bound $A_\text{max}=\frac{R}{\beta-\nu}$ of the attracting set. We have 
\[
 f(A_\text{max})=-\frac{\beta (N-1)A_\text{max}}{2 N}(1-e^{-\beta^2 \frac{A_\text{max}}{N}}) +\frac1{\beta}(1-e^{-\beta^2 \frac{A_\text{max}}{N}})-\frac{\beta A_\text{max}}{N}<0
 \]
because $1-e^{-x}<x$ for all $x>0$. Therefore, the function $f$ must have a zero $A^\text{FP}$ in the interval $(0,A_\text{max})$.
 
In order to prove uniqueness, we successively compute the first and second derivatives of $f$ to obtain
\[
f''(A)=\frac{\beta\nu}{N}\left(2(\nu -\beta)-\frac{\beta\nu}{N}\left(\frac{(N+1)(R+\nu A)}{2 N}-\beta A\right)\right)e^{-\beta \frac{R+\nu A}{N}}.
\]
Since $\beta> \nu>\nu\frac{N+1}{2N}$, the expression inside the parenthesis is increasing, {\sl viz.}\ the sign of $f''$ can only change once in $[0,A_\text{max}]$, from negative to positive. Together with the facts that $f(0)>0$ and $f(A_\text{max})<0$, this implies that $f$ crosses 0 only once in this interval. 
\medskip

\noindent
{\em (ii)} As in \cite{FT11}, the finite-dimensional return dynamics (that takes place in the neighbourhod of $(\mathbf x^\text{FP},A^\text{FP})$) can be regarded as the composition of the following variation of the firing maps introduced at the begining at the proof of Theorem \ref{STABPARTSYNC} 
\begin{equation}
x'_i=\left\{\begin{array}{ccl}
R+\nu A-x_{N-1}&\text{if}&i=1\\
x_{i-1}-x_{N-1}&\text{if}&i\in\{2,\cdots,N-1\}\\
\end{array}\right.\quad\text{and}\quad A'=\phi_m^{x_{N-1}}(A),
\label{FIRINGMAP}
\end{equation}
More precisely, the local stability of $(\mathbf x^\text{FP},A^\text{FP})$ can be obtained from the spectrum of the product of the derivatives of this map along the elements of the corresponding periodic orbit. However, since $(\mathbf x^\text{FP},A^\text{FP})$ is a fixed point of this firing map, one actually has to evaluate the spectrum of $M_N^N$ where $M_N$ is the derivative of \eqref{FIRINGMAP} evaluated at $(\mathbf x^\text{FP},A^\text{FP})$ with 
\[
m(t)=m(0^+)-t=\frac1{N}\left(\sum_{i=1}^{N-1}x_i+R+\nu A\right)-t\ \text{for}\ \in (0,x_{N-1}].
\]
From \eqref{Eq:EvforA} and $m(t)=m(0^+)-t$, we have
\[
\phi_m^t(A)=Ae^{-\beta t}+\left(\frac{m(0^+)}{\beta}+\frac1{\beta^2}\right)(1-e^{-\beta t})-\frac{t}{\beta},
\]
which yields
\[
\partial_{x_i} \phi_m^{x_{N-1}}(A)=\left\{\begin{array}{ccl}
\frac{1-e^{-\beta x_{N-1}}}{\beta N}&\text{if}&i\in\{1,\cdots ,N-2\}\\
\left(m(0^+)+\frac1{\beta}-\beta A\right)e^{-\beta x_{N-1}} -\frac1{\beta}+\frac{1-e^{-\beta x_{N-1}}}{\beta N}&\text{if}&i=N-1
\end{array}\right.
\]
and
\[
\partial_{A} \phi_m^{x_{N-1}}(A)=e^{-\beta x_{N-1}}+\frac{\nu(1-e^{-\beta x_{N-1}})}{\beta N}.
\]
Evaluating these quantities at $(\mathbf x^\text{FP},A^\text{FP})$ and expanding in $\frac1{\beta}$, one gets that the derivative $M_N$ writes
\[
M_N=U_N+\frac1{\beta} V_N+o(\frac1{\beta}),
\]
where $U_N$ and $V_N$ are the following $N\times N$ matrices
\[
U_N=\left(\begin{array}{cccccc}0&\cdots&\cdots&0&-1&\nu\\1&0&\cdots&0&-1&0\\0&1&\ddots&\vdots&\vdots&\vdots\\\vdots&&\ddots&\vdots&\vdots&\vdots\\0&\cdots&0&1&-1&0\\0&\cdots&\cdots&\cdots&\cdots&0\end{array}\right)\quad\text{and}\quad 
V_N=\left(\begin{array}{cccccc}0&\cdots&\cdots&\cdots&\cdots&0\\\vdots&&&&&\vdots\\\vdots&&&&&\vdots\\\vdots&&&&&\vdots\\\\0&\cdots&\cdots&\cdots&\cdots&0\\\frac1{N}&\cdots&\cdots&\frac1{N}&-\frac{N-1}{N}&\frac{\nu}{N}\end{array}\right).
\]
Accordingly, we have
\[
M_N^N=W_N+o(\frac1{\beta})\quad\text{where}\quad W_N:=U_N^N+\frac1{\beta}\sum_{k=1}^{N-1}U_N^kV_NU_N^{N-1-k}.
\]
For $N\in\{2,\cdots ,5\}$, we have checked that $W_N$ writes
\[
W_N=\left(\begin{array}{cccccc}1-\frac{(N-1)\nu}{N\beta}&\frac{\nu}{N\beta}&\cdots&\cdots&\frac{\nu}{N\beta}&-\nu(1-\frac{\nu}{N\beta})\\\frac{\nu}{N\beta}&1-\frac{(N-1)\nu}{N\beta}&\frac{\nu}{N\beta}&\cdots&\frac{\nu}{N\beta}&-\nu(1-\frac{\nu}{N\beta})\\\frac{\nu}{N\beta}&\frac{\nu}{N\beta}&\ddots&\ddots&\vdots&\vdots\\\vdots&&\ddots&\ddots&\frac{\nu}{N\beta}&\vdots\\\frac{\nu}{N\beta}&\cdots&\cdots&\frac{\nu}{N\beta}&1-\frac{(N-1)\nu}{N\beta}&-\nu(1-\frac{\nu}{N\beta})\\\frac{1}{N\beta}&\cdots&\cdots&\cdots&\frac{1}{N\beta}&-\frac{(N-1)\nu}{N\beta}\end{array}\right).
\]
The eigenvectors of $W_N$ and its eigenvalues can be readily obtained. Firstly, notice that 
\[
U_N(\nu,\cdots,\nu,1)^T=V_N(\nu,\cdots,\nu,1)^T=0,
\]
hence $W_N(\nu,\cdots,\nu,1)^T=0$ too. Moreover, $W_N$ clearly has the following $N-2$ eigenvectors
\[
(1,-1,0,\cdots,0)^T,\ (1,0,-1,0,\cdots,0)^T,\cdots,\ (1,0,\cdots,0,-1,0,0)^T,\ (1,0,\cdots,0,-1,0)^T,
\]
(or explicitly $(v_k)_i=\delta_{i,1}-\delta_{i,k}$ for $i,k\in\{2,\cdots ,N-1\}$) and the corresponding eigenvalue is equal to $1-\frac{\nu}{\beta}$ in each case. Finally, the remaining eigenvector writes $(N\beta-\nu,0,\cdots ,0,1)^T$ and the corresponding eigenvalue is also equal to $1-\frac{\nu}{\beta}$. Therefore, the spectrum of $W_N$ lies inside the unit disk; hence so does the spectrum of $M^N_N$ when $\beta$ is sufficiently large. Statement {\em (ii)} is proved. 
\hfill $\Box$

\section{Concluding remarks}
In this paper, we have presented an extended mathematical analysis of the dynamics of the model introduced in \cite{MHT14} of DF oscillators coupled via a common activator, with emphasis on the stability with respect to out-of-sync perturbations depending on the delay in activator synthesis. After a study of the basic properties of the flow associated with the equations \eqref{ODE1}-\eqref{ODE2}, the analysis has firstly considered the case of fully synchronized trajectories and then has adressed arbitrary partially synchronized periodic orbits. 

The most significant outcome of this endeavour is that the stability of periodic orbits abruptly changes when the delay is switched on. From an unstable solution, the orbit immediately becomes asymptotically stable with respect to small perturbations that smear its clusters. 

While this change of behaviour appears to be spectacular, it can be readily apprehended from the general criteria of Section \ref{S-PARTSYNC} (which themselves can be intuited from the expression \eqref{Eq:ReturnMap} of the return map) together with the profile of the activator concentration close to firing, in the periodic orbit. More precisely, Lemma \ref{STABCRIT} about stability requires that $A^\text{FP}(\cdot)$, as a function of time, be decreasing immediately before firing. Lemma \ref{INSTABCRIT} about instability needs that the derivative of this function be positive and sufficiently large immediately after firing. 

Notice that, according to the original equation \eqref{ODE1}--\eqref{ODE2}, both properties should be commmon features of the periodic orbits in this system, as illustrated for the synchronized trajectory in the left panel of Fig.\ \ref{Fig:FlowPlusReturnMap}. Indeed, the variation of the vector field acting on $A$ must be locally minimal immediately before any firing (because $m$ decreases between two consecutive firings) and firings trigger sudden increases of this vector field.

Finally, we believe that these stylized features and the accompanying rapid change in stability extend to more general, smooth models of DF oscillators with similar ingredients, especially when firings occur on very short time scales and the corresponding reset values are impacted by some delay in the activator synthesis. Such extensions could be the subject of future studies.
\appendix

\section{Lipschitz continuity of the return map}\label{A-LIPCONT}
Let $N\in\N$ and $\tau>0$ be given. Consider an initial datum $(\mathbf x,A_{[-\tau,0]})$ in the section $x_N=0$, $\|A|_{[-\tau,0]}\|_0<A_\text{max}$ and $\text{max}_{\mathbf x}<x_N(0^+)=t_\text{R}=R+\nu A(-\tau)$, so that the trajectory lies in the attracting set $Q$ and the order in which the oscillators fire remains the same starting from $t=0$. Assume also that $A_{[-\tau,0]}$ is Lipschitz continuous, with maximal Lipschitz constant $\beta A_\text{max}$ when in $Q$ (see the proof of Claim \ref{CLAIMORDER} above).

Given another datum $(\mathbf x',A'_{[-\tau,0]})$ with the same constraints, let $t'_\text{R}$ be the first return time to the section $x_N=0$. We have the following statement
\begin{Lem}
In addition to the assumptions above on $(\mathbf x,A_{[-\tau,0]})$ and $(\mathbf x',A'_{[-\tau,0]})$, assume that 
\[
\text{max}_{\mathbf x}<t'_\text{R}\ \text{if}\ t'_\text{R}\leq t_\text{R}\quad\text{and}\quad 
\text{max}_{\mathbf x'}<t_\text{R}\ \text{if}\ t_\text{R}\leq t'_\text{R}.
\]
Then, there exists $L\in\R^+$ (which is independent of the data) such that we have 
\[
\max\left\{\|\mathbf x(t_\text{R})-\mathbf x'(t'_\text{R})\|_{N-1},\|A|_{[t_\text{R}-\tau,t_\text{R}]}-A'|_{[t'_\text{R}-\tau,t'_\text{R}]}\|_0\right\}\leq L \max\left\{\|\mathbf x-\mathbf x'\|_{N-1},\|A|_{[-\tau,0]}-A'|_{[-\tau,0]}\|_0\right\}.
\]
\label{LIPCONT}
\end{Lem} 
\begin{proof}
{\em Estimate of $\|A|_{[t_\text{R}-\tau,t_\text{R}]}- A'|_{[t'_\text{R}-\tau,t'_\text{R}]}\|_0$:} We have $t_\text{R}-t'_\text{R}=\nu(A(-\tau)-A'(-\tau))$. Using \eqref{Eq:EvforA}, we consider the following decomposition similar to the one in the proof of statement {\em (ii)} of Proposition \ref{MAINRES}
\begin{align*}
A(t_\text{R}-t)-A'(t'_\text{R}-t)=&A(t_\text{R}-t)(1-e^{-\beta(t'_\text{R}-t_\text{R})})+\left(A(0)-A'(0)+\int_0^{t_\text{R}\wedge t'_\text{R}-t}e^{\beta s}(m(s)-m'(s))ds\right.\\
&+\left.\int_{t_\text{R}\wedge t'_\text{R}-t}^{t_\text{R}\vee t'_\text{R}-t} e^{\beta s}\overline{m}(s)ds\right)e^{-\beta (t'_\text{R}-t)}
\end{align*}
for $t\in [-\tau,0]$, where 
\[
\overline{m}:=\left\{\begin{array}{ccl}
m&\text{if}&t_\text{R}>t'_\text{R}\\
-m'&\text{if}&t_\text{R}<t'_\text{R}
\end{array}\right.
\]
That $A(t),A'(t)$ are uniformly bounded implies the same property for $t'_\text{R}$ and $m,m'$. Together with the expression of $t_\text{R}-t'_\text{R}$ above, this implies that the modulus of the first, second and last terms in the above decomposition can be controlled by $\|A|_{[-\tau,0]}-A'|_{[-\tau,0]}\|_0$. For the remaining term, we need to control $|m(s)-m'(s)|$ for $s\in [0,t_\text{R}\wedge t'_\text{R}-t]$. Notice that each oscillator fires at most once in this interval. Hence, we have
\[
x_i(s)-x'_i(s)=\left\{\begin{array}{ccl}
x_i-x'_i&\text{if}&s\leq x_i\wedge x'_i\\ 
\nu (A(x_i-\tau)-A'(x'_i-\tau))+x_i-x'_i&\text{if}&s> x_i\vee x'_i
\end{array}\right.
\]
Accordingly, outside the intervals $[x_i\wedge x'_i,x_i\vee x'_i]$, the quantity $|m(s)-m'(s)|$ is certainly bounded by 
\begin{align*}
&\frac{\nu}{N}\sum_{i=1}^{N-1}|A(x_i-\tau)-A(x'_i-\tau)|+\frac{\nu}{N}\sum_{i=1}^{N-1}|A(x'_i-\tau)-A'(x'_i-\tau)|+\frac{1}{N}\sum_{i=1}^{N-1}|x_i-x'_i|\\
&\leq L_1\|\mathbf x-\mathbf x'\|_{N-1}+\nu \|A|_{[-\tau,0]}-A'|_{[-\tau,0]}\|_0
\end{align*}
for some $L_1\in\R^+$ sufficiently large, where, in addition to the fact that $A$ is Lipschitz continuous on $\R^+$, for those $x_i,x'_i\in (0,\tau]$, the second inequality also relies on the assumption that $A|_{[-\tau,0]}$ is Lipschitz continuous. Moreover, inside the intervals $[x_i\wedge x'_i,x_i\vee x'_i]$, the quantity $|m(s)-m'(s)|$ is uniformly bounded because the corresponding trajectories are in $Q$. There are at most $N-1$ such intervals whose length is bounded by $\|\mathbf x-\mathbf x'\|_{N-1}$. Altogether, the arguments here prove the existence of $L_2,L_3\in \R^+$ (which do not depend on $(\mathbf x,A_{[-\tau,0]})$ and $(\mathbf x',A'_{[-\tau,0]})$ when in $Q$) such that 
\[
\|A|_{[t_\text{R}-\tau,t_\text{R}]}- A'|_{[t'_\text{R}-\tau,t'_\text{R}]}\|_0\leq L_2 \|\mathbf x-\mathbf x'\|_{N-1}+L_3 \|A|_{[-\tau,0]}-A'|_{[-\tau,0]}\|_0.
\]

\noindent
{\em Estimate of $\|\mathbf x(t_\text{R})-\mathbf x'(t'_\text{R})\|_{N-1}$:} Assume that $t_\text{R}\leq t'_\text{R}$, the other case can be treated similarly. Consider the decomposition
\[
x_i(t_\text{R})-x'_i(t'_\text{R})=x_i(t_\text{R})-x'_i(t_\text{R})+x'_i(t_\text{R})-x'_i(t'_\text{R}).
\]
The assumptions $\text{max}_{\mathbf x},\text{max}_{\mathbf x'}<t_\text{R}$ imply that all oscillators in both trajectories must have fired once when at instant $t_\text{R}$. Accordingly, the first difference can be controlled using the same estimate as above in the case $s>x_i\vee x'_i$. Moreover, all oscillators with $x'_i>0$ have not fired a second time at instant $t'_\text{R}$. Hence, we have $x'_i(t_\text{R})-x'_i(t'_\text{R})=t'_\text{R}-t_\text{R}$ which is also well under control. Altogether, this proves that a similar estimate as for $\|A|_{[t_\text{R}-\tau,t_\text{R}]}- A'|_{[t'_\text{R}-\tau,t'_\text{R}]}\|_0$ holds for $\|\mathbf x(t_\text{R})-\mathbf x'(t'_\text{R})\|_{N-1}$. 
\end{proof}

\end{document}